\documentclass[11pt]{article}

\usepackage{amsmath,amsthm,amssymb,graphicx,color}

\usepackage{setspace}
\usepackage[margin=1in]{geometry}
\usepackage{fullpage}

\usepackage{epsfig,amsfonts,latexsym}
\usepackage{enumerate}
\usepackage{multirow}
\usepackage{float}
\usepackage{bbm}
\usepackage{caption}
\usepackage{subcaption}
\usepackage{url}
\usepackage{hyperref}
\usepackage{geometry}
\usepackage{lipsum}
\usepackage{enumitem}
\usepackage{bm}
\usepackage{algpseudocode}
\usepackage{courier}
\usepackage{hyperref}
\hypersetup{colorlinks,
citecolor=blue,
urlcolor=black,
linkcolor=black
}

\newcommand{\bigperp}[2]{%
  \vcenter{
    \math\hbox{\scalebox{\ifx#1\displaystyle2.1\else1.5\fi}{$#1\perp$}}
  }%
}

\newcommand{\argmin}{\operatornamewithlimits{argmin}}
\newtheorem{thm}{Theorem}[section]
\newtheorem{corollary}[thm]{Corollary}

\newtheorem{lemma}[thm]{Lemma}

\newtheorem{remark}[thm]{Remark}

\newtheorem{Lemma}[thm]{Lemma}

\theoremstyle{definition}

\newcommand\numberthis{\addtocounter{equation}{1}\tag{\theequation}}

\newcommand{\vp}{{\varphi}}
\newcommand{\ep}{{\epsilon}}
\newcommand{\calA}{{\mathcal{A}}}
\newcommand{\calB}{{\mathcal{B}}}
\newcommand{\calC}{{\mathcal{C}}}

\newcommand{\calG}{{\mathcal{G}}}
\newcommand{\calH}{{\mathcal{H}}}

\newcommand{\calL}{{\mathcal{L}}}
\newcommand{\calK}{{\mathcal{K}}}
\newcommand{\calM}{{\mathcal{M}}}
\newcommand{\calN}{{\mathcal{N}}}

\newcommand{\calR}{{\mathcal{R}}}

\newcommand{\calV}{{\mathcal{V}}}

\newcommand{\diag}{{\rm diag}}

\newcommand{\real}{\mathbb{R}}
\newcommand{\E}{\mathbb{E}}
\newcommand{\B}{\mathbb{B}}
\newcommand{\PP}{\mathbb{P}}

\newcommand{\grad}{{\nabla}}

\newcommand{\Q}{Q}
\newcommand{\Qinv}{\Theta}

\newcommand{\minimize}{\mathop{\mathrm{minimize}}}

\def\param{\varsigma}

\allowdisplaybreaks

\newcount\Comments  
\definecolor{darkgreen}{rgb}{0,0.5,0}
\definecolor{purple}{rgb}{1,0,1}
\newcommand{\kibitz}[2]{\ifnum\Comments=1\textcolor{#1}{#2}\fi}

\usepackage[round]{natbib} 

\begin{document}


\begin{center}
\LARGE Kinematic Formula for Heterogeneous Gaussian Related Fields
\end{center}

%



\newcommand\blfootnote[1]{%
  \begingroup
  \renewcommand\thefootnote{}\footnote{#1}%
  \addtocounter{footnote}{-1}%
  \endgroup
}

\medskip

\begin{center}
{\large Snigdha Panigrahi \ \ \ Jonathan Taylor \ \ \ Sreekar Vadlamani}
\end{center}



\newcommand{\Addresses}{{
  \bigskip
  \footnotesize
  
Snigdha Panigrahi, \textsc{Department of Statistics, Stanford University.
    }\par\nopagebreak
  \textit{E-mail address}, Snigdha Panigrahi: \texttt{snigdha@stanford.edu}

  \medskip

  Jonathan Taylor, \textsc{Department of Statistics, Stanford University.
  }\par\nopagebreak
  \textit{E-mail address},Jonathan Taylor : \texttt{jonathan.taylor@stanford.edu}

  \medskip

  Sreekar Vadlamani, \textsc{TIFR- Center for Applicable Mathematics, Bangalore
  }\par\nopagebreak
 \textit{E-mail address}, Sreekar Vadlamani: \texttt{sreekar@math.tifrbang.res.in}

}}


\begin{abstract}
We provide a generalization of the Gaussian Kinematic Formula (GKF) in \cite{Taylor06} for multivariate, heterogeneous Gaussian-related fields. The fields under consideration, $f=F\circ y$, are non-Gaussian fields built out of smooth, independent Gaussian fields $y={(y_1,y_2,..,y_K)}$ with heterogeneity in distribution amongst the individual building blocks. Our motivation comes from potential applications in the analysis of Cosmological Data (CMB). 
Specifically, future CMB experiments will be focusing on polarization data, typically modeled as isotropic vector-valued Gaussian related fields with independent, but non-identically distributed Gaussian building blocks; this necessitates such a generalization. Extending results \cite{Taylor06} to these more general Gaussian relatives with distributional heterogeneity, we present a generalized Gaussian Kinematic Formula (GKF). The GKF in this paper decouples the expected Euler characteristic of excursion sets
into Lipschitz Killing Curvatures (LKCs) of the underlying manifold and certain Gaussian Minkowski Functionals (GMFs). These GMFs arise from Gaussian volume expansions of ellipsoidal tubes as opposed to the usual tubes
in the Euclidean volume of \textit{tube formulae}. The GMFs form a main contribution of this work that identifies this tubular structure and a corresponding volume of tubes expansion in which the GMFs appear. 
\end{abstract}









\section{Smooth random fields and integral geometry}
\label{intro}

Since the work of \cite{Adler81} and \cite{worsley1994local}, the study of smooth
(usually Gaussian) random fields has exposed a very nice connection
between properties of the excursion sets of the random fields
and integral geometric properties of the parameter space of the field
\cite{worsley1994local}. In more recent work, \cite{Taylor06} the 
integral geometric
story has been extended to also include integral geometric
properties of the marginal distribution (assumed constant) of the
random field. This connection has been dubbed a Gaussian Kinematic Formula.
In this work, we extend the GKF to a larger class of random fields, relaxing
the assumption of identical distribution. Before stating
our main result, we recall some classical quantities in integral geometry
as well as earlier work in smooth random fields.

\subsection{Kinematic Formulae}

Perhaps the canonical example of
integral geometric formulae are the {\it kinematic fundamental formulae} (KFF), having been applied in areas
such as biology, mineralogy and metallurgy (see \cite{Santalo} and references therein). Kinematic fundamental formulae are equalities establishing relationships between some averaged 
global geometric features of all possible intersections of two given {\it bodies}, and the global geometric
quantities of individual bodies. In this sense, they can be viewed
as generalizations of Buffon's needle problem.

In order to formulate the KFF, we dwell on the global geometric characteristics called the 
{\it Lipschitz-Killing curvatures} (LKCs) or intrinsic volumes, which are at the heart of such formulae. Given a $d$-dimensional 
smooth manifold $M$, LKCs are $(d+1)$ intrinsic, geometric functionals denoted as $\{\bm\calL_k(M)\}_{k=0}^d$ 
such that they satisfy following properties:
\begin{itemize}
\item each $\bm\calL_k$ for $k=0,\ldots,d$ is a finitely additive set functional;
\item for any $\lambda >0$, and a {\it nice} set $A$, we have $\bm\calL_k(\lambda A) = \lambda^k\bm\calL_k(A)$, for all 
$k=0,\ldots, \text{dim}(A)$;
\item all $\bm\calL_k$ are rigid motion invariant {\it i.e.}, for any {\it nice} set $A$, and any rigid motion $g$, writing
$gA = \{gx:\,x\in A\}$ we have $\bm\calL_k(gA) =\bm \calL_k(A)$, for all $k=0,\ldots, \text{dim}(A)$;
\item each $\bm\calL_k$ is continuous (we refer the reader to \cite{Klain-Rota} for more details).
\end{itemize}
A simple example of encountering LKCs is the Steiner-Weyl tube formula (itself a special case of the KFF)
$${\cal H}_k(\text{tube}(M,\epsilon))=\sum_{j=0}^{\text{dim} M}\epsilon^{k-j}\text{Vol}(B_{\real^k}(1))\bm{\calL}_j(M),$$
where $B_{\real^k}(1)$ is the unit ball in $\real^k$ and $\text{tube}(M,\epsilon)=\{y\in \real^k: \inf_{x\in M}\|x-y\|\leq \epsilon\}$. This gives the volume of an $\epsilon$-tubular neighborhood around a wide class of sets $M\subset \real^k$.

Equipped with the above definition/characterization of LKCs, we now state the most general Euclidean KFF,
as it appears in \cite{RFG}. Let $M_1$ and $M_2$ be two nice sets in 
$\real^d$, and let $\calG_d$ be the group of rigid motions on $\real^d$, then
\begin{equation}\label{eqn:KFF}
\int_{\calG_d} \bm\calL_m\left( M_1 \cap gM_2\right)\,\nu(dg) 
= \sum_{j=0}^{d-m} \frac{s_{m+1}\, s_{d+1}}{s_{m+j+1} \, s_{d-j+1}}\bm\calL_{m+j}(M_1)\bm\calL_{d-j}(M_2)
\end{equation}
where $\nu$ is the normalised Haar measure on $\calG_d$, and $s_k$ denotes the surface area of a unit 
ball in $\real^k$. 

\begin{remark}
For a definition of {\it nice} sets we refer to \cite{RFG,Brocker-Kuppe}.
\end{remark}

KFFs have a rich history, and we refer the reader to 
\cite{RFG, Brocker-Kuppe, Klain-Rota, Schneider-Weil-92, Schneider-Weil-08}, 
and references therein, for an exhaustive account. 
All the available proofs of KFF are delicate, and rely heavily on various invariances available in the
setup, like the invariance of Lebesgue measure under $\calG_d$ plays a crucial role.
A natural question then, one may ask, is if such integral formulae are exclusive only to Euclidean space
with Lebesgue measure.


\subsection{Excursion sets of Gaussian processes}

Interestingly, another problem which bears striking resemblance with Buffon's needle problem
got many mathematicians interested. In order to exhibit the similarity, we can hypothesize the problem as 
having to sample a {\it random} path, instead of throwing a needle, and then counting the number
of crossings of this random curve with a fixed line.

In 1940s, \cite{kac} and \cite{rice} solved this problem analytically with some basic regularity assumptions. 
Using a clever argument to count the number of crossings of a given function, Kac and Rice, working independently, obtained
a compact expression for the mean number of crossings of a random algebraic function under some mild
regularity conditions.

%
%
 Revisiting \eqref{eqn:KFF}, the case $m=0$ gives rise to the expected Euler-Poincar\'e characteristic (called Euler characteristic in the rest of this paper) denoted as $\bm\chi(.)$. The expected Euler characteristic of excursion sets of smooth random fields $f$ on a $C^3$ domain $M$ defined as
$$\E[\bm\chi\{t\in M: f(t)\geq u\}],$$
has been studied extensively in \cite{Adler81, adler2000excursion, worsley1994local, worsley1995boundary, taylor2003euler, Taylor06}. The  derivation of expected Euler characteristic for stationary Gaussian random fields dates back to \cite{Adler81}, with generalizations to $\chi^2$, F and t-fields and to higher dimensions in \cite{worsley1994local, worsley1995boundary}. \cite{Adler81}, and later, \cite{taylor2003euler}, generalized the 
counting technique in  \cite{kac,rice} to the multiparameter case.
These papers set the stage for what is now called {\it expectation metatheorem} 
which can be viewed as quite general form of Kac-Rice formula
(see \cite{RFG} for details). The expectation metatheorem can be stated as: let $n, k\ge 1$, and 
$$G=(G_1,\ldots, G_n) \text{ and } H=(H_1,\ldots,H_k),$$ be two $\real^n$ and $\real^k$ valued a.s. continuous random fields defined
on an $n$-dimensional, parameter space $T\subset\real^n$ such that $T$ is smooth and compact. Let $U$ be
an open subset of $\real^k$ such that the Hausdorff dimension of the boundary $\partial U= \overset{-}{U} \setminus U$ 
is $(k-1)$, then writing $$N_u(G,H;T,U) = \{x\in T:\, G(x)=u, \text{ and } H(x)\in U\},$$ under some regularity conditions 
(see \cite[Theorem 11.2.1]{RFG}), we have
\begin{equation}\label{eqn:meta}
\E\left( N_u(G,H;T,U)\right) = \int_{T} \E\left\{|\det\grad G|\, 1_{U}(H(x))\biggr \vert G(x)=u\right\} p_x(u)\, dx,
\end{equation}
where $p_x$ is density of the random variable $G(x)$.

In \cite{taylor2003euler}, the expected Euler characteristic for centered and unit variance, smooth Gaussian random fields $f$ 
on a smooth manifold $M$, based on the expectation meta theorem, was shown as a decoupling into LKCs of the manifold 
$M$ and coefficients that are products of Hermite polynomials with the standard Gaussian density. That is
$$\E[\bm\chi(M\cap f^{-1}[u,\infty))]=\sum_{j=0}^{n} \bm{\calL}_{j}(M)\rho_j(u),$$
with 
\[ \rho_j(u)=\begin{cases} 
      1-\Phi(u) & j= 0 \\
      \cfrac{1}{(2\pi)^{(j+1)/2}}H_{j-1}(u)\exp(-u^2/2) & j\geq 1. 
       \end{cases}
\]

\subsection{Gaussian integral geometry and the GKF}

\cite{Taylor06} provided geometric meaning to the coefficients which appeared earlier
in \cite{Adler81, taylor2003euler} via a Gaussian tube formula, and also extended the earlier calculations of \cite{Adler81,worsley1994local}
to  a class of multivariate non-Gaussian random fields, which led to the formulation of {\it Gaussian kinematic formula} (GKF). 

In the case of Gaussian random fields studied in \cite{taylor2003euler}, the EC densities $ \rho_j(u)$ are seen to match up to a factor of $(2\pi)^{-j/2}$ with the coefficients $$\bm{\calM}_j^{\gamma^{\real^1}}([u,\infty))=(2\pi)^{-1/2}H_{j-1}(u)\exp(-u^2/2)$$
arising in a Gaussian tubular volume expansion $\gamma^{\real^1}(\text{tube}([u,\infty),\epsilon))=\gamma^{\real^1}([u-\epsilon,\infty))$.

The GKF formula, more generally, can be stated as a decoupling of the mean LKCs of excursion sets into LKCs of $M$ and GMFs that are seen in the tube formula. Suppose $f=F \circ y$
 whose components are smooth, independent and marginally
stationary with marginal law $N(0,1)$. Then,
\begin{equation}
\label{eq:gkfold}
\E[\bm\calL_m(M\cap f^{-1}[u,\infty))]=\sum_{j=0}^{\text{dim }M-m} {{m+j}\brack {j}}(2\pi)^{-j/2}\bm{\calL}_{m+j}(M)\bm\calM_{j}^{\gamma_\real^K}(F^{-1}[u,\infty)),
\end{equation}
with ${{m+j}\brack {j}}=\cfrac{(m+j)!\nu_n}{m! j!\nu_m\nu_j}$ and
$\nu_j=\cfrac{\pi^{j/2}}{\Gamma(n/2+1)}$ the volume of a unit ball in $\real^n.$
\medskip

The above, re-derived in \cite{taylor2009gaussian} can be viewed as recasting \eqref{eqn:KFF} in the form of a KFF over Gaussian function space.
The Gaussian Minkowski Functionals in \cite{taylor2009gaussian} are defined
 implicitly in a generalization of the Steiner-Weyl formula
$$
\begin{aligned}
\gamma_{\real^K}(\text{tube}({\cal K}, \epsilon)) 
&= \gamma_{\real^K}({\cal K} \oplus B_{\real^K}(\epsilon)) = \sum_{j \geq 0} 
\frac{\epsilon^j}{j!} 
\bm\calM_{j}^{\gamma_\real^K}({\cal K}).
\end{aligned}
$$

\subsection{Extension of GKF}

Formulae for the expected Euler characteristic
of a smooth Gaussian field has already found many important applications in the analysis of cosmological data (CMB), see \cite{CMB-GKF}, \cite{planck2014planck}, \cite{ade2015planck} for more details on this.  
Isotropic vector-valued Gaussian related fields with independent, but non-identically distributed Gaussian building blocks can arise in modeling of polarization data in CMB experiments, necessitating a generalization of GKF in \cite{Taylor06} to multivariate, heterogeneous, Gaussian related random fields. Motivated by applications in CMB experiments, our goal in this paper is to derive a GKF for heterogenous, Gaussian relatives $f=F\circ y$, constructed out of Gaussian fields ${y}={(y_1,y_2,...,y_K)}$ with distributional heterogenity in individual components. That is, the component Gaussian fields are marginally stationary and independent, but non-identically distributed. Specifically, the building blocks are non-identical in distribution in the following sense: 
the gradient field ${\grad y}$ has separable covariance structure 
$$\diag(\lambda_1,...,\lambda_K)\otimes I \overset{\text{def}}{=} D \otimes I,$$ 
and induces conformal Riemannian metrics (with constant conformal factor) on the manifold $M$. As such our prototypical model in this work has
$M=S(\real^3)$ and each $y_i$ is isotropic on the sphere with
possibly different spectral measures.

 The main theorem of the paper derives a (GKF) for the expected Euler characteristic of the excursion sets of such fields $f$, yielding the result
\begin{equation}
\label{eq:gkfnew}
\E[\bm\chi(M\cap f^{-1}[u,\infty)]=\E[\bm\chi(M\cap y^{-1}\calK)]=\sum_{j=0}^{n} (2\pi)^{-j/2}\bm{\calL}_j(M)\bm\calM_{j}^{\gamma_\real^K,D}(\calK),
\end{equation}
 with $\calK=F^{-1}[u,\infty)$, $\bm{\calL}_j(M)$
the LKCs of $M$, and $\bm\calM_{j}^{\gamma_\real^K,D}(\calK)$
the Gaussian Minkowski functionals. 

In comparing \eqref{eq:gkfold} to \eqref{eq:gkfnew} the reader will notice
that we have introduced a parameter $D$ to the Gaussian Minkowski functionals
in \eqref{eq:gkfnew}.
Define an ellipsoidal tube with as
\begin{equation}
\label{eq:ellipse:tube}
T^{D}(\calK,\ep)=\calK\oplus  B_{D,\real^K}(\ep),
\end{equation}
with
\begin{equation}
\label{eq:ellipse}
B_{D,\real^K}(\epsilon) = \left\{w \in \real^K: w^T D^{-1} w\leq \epsilon^2 \right\}
\end{equation}
recalling that $D=\text{diag}(\lambda_1, \dots ,\lambda_K)$.

The GMFs above are implicitly defined as terms in a Taylor series expansion for the Gaussian volume of $T^{D}(\calK,\ep)$ in terms of integrals on the boundary of $\calK$ given by
\begin{eqnarray}
\label{gaussian expansion}
&& \gamma_{\real^K}(T^{D}(\calK,\ep))= \gamma_{\real^K}({\cal K}) + \sum_{l=1}^{\infty}\dfrac{\ep^l}{l!}\bm\calM_{l}^{\gamma_{\real^K},D}(\calK).
\end{eqnarray}
Note then that the original GMFs in \eqref{eq:gkfold} simply correspond
to the case $D=I$.

\subsection{Outline of the paper}

The main result is stated formally in Section \ref{KFF}, and is established as follows. We show in Section \ref{PP} that the expected Euler Characteristic for the Gaussian related fields under consideration can be expanded in terms of EC densities, computed as integrals with respect to standard Gaussian measure and Lipschitz Killing curvatures (LKC). The LKCs are derived from the Riemannian curvature induced by the base spatial metric $g$ on $M$. Section \ref{IEC} is devoted to obtaining explicit integral representations of EC densities by carefully modifying the tools developed in \cite{Taylor06} to be adapted to our setting. The result follows with the observation that the coefficients in volume expansions of the ellipsoidal tubes considered in Section \ref{GP} and the integral representation of $\{\tilde{\rho}_{j}(F,u); j=1,2,..,n\}$ in Section \ref{IEC} match up to a factor of $(2\pi)^{-j/2}$, leading to an analogous GKF for heterogeneous Gaussian related fields. Finally, in Section \ref{application}
we conclude with an application of our results to the study of cosmic microwave background radiation data.


\section{GKF for heterogeneous Gaussian fields}
\label{KFF}

In this section, we formally state a generalization of the GKF in \cite{Taylor06} which somewhat relaxes distributional assumptions. The essence of such a result lies in the decoupling of spatial information and distributional information of the random field. We begin with
formally describing our assumptions on the $\real^K$ valued field $y$.

\subsection{Set up and assumptions} 
\label{assumption0}

We describe the heterogeneous Gaussian related fields $f=F\circ y$ under consideration in the paper by listing a set of assumptions on the individual Gaussian building blocks $y$ and the function $F$. 

The Gaussian building blocks ${y}={(y_1,..,y_K)}$ in particular are assumed to satisfy the following assumptions:
\begin{enumerate}[label={\textnormal{\textbf{(\Alph*)}}},itemsep=2pt,parsep=2pt]
\item Marginal Stationarity and Independence: ${(y_1,..,y_K)}$ are individually real-valued, mean $0$, unit variance, independent random fields on manifold $M$. \label{item:1}
\item Separability of gradient field: We assume a separable structure for the covariance of the gradient field \label{item:2}
\begin{equation}
\label{KP}
{\text{Cov}({\grad y})={D \otimes I} .}\nonumber
\end{equation}
where ${ D=\diag(\lambda_1,...,\lambda_K)}$. Here $D$ represents the covariance amongst the random fields, while $I$ denotes the spatial covariance.
\item Metric Conformity: Each $y_k$ induces a metric $g_k$ on the manifold $M$ such that
$$g_{i,j}^{k}=g^{k}(E_i,E_j)=\lambda_{k}g(E_i,E_j),$$
where $\{E_i\}$ being an orthonormal frame field with respect to base spatial metric $g$ and $\lambda_{k}$ being the second spectral moment of the field $y_k$. \label{item:3}
\item Regularity: The tuple $ Y_k={(y_k(t),\grad y_k(t),\grad^2 y_k(t))}$ should satisfy:
$$\PP(\sup_{u\in \calB(t,h)}\|{Y_k}(t)-{Y_k}(u)\|_2>\epsilon)=o(h^n),$$
for any $\epsilon>0$,
for the metric $\|\|_2$ defined as :$$\|{Y_k}(t)\|_2=|y_k(t)| +\|\grad y_k(t)\|_{\real^n} +\|\grad^2 y_k(t)\|_{\otimes^2\real^n},$$
for ball $\calB(t,h)$ around $t$ with radius $h$.  \label{item:4} 
\end{enumerate}
\begin{remark}
Under the assumption of marginal stationarity, we have independence of the gradient field from the field and hessian evaluated at a point $t$, that is
$$\grad y_k(t) \perp (y_k,\grad^2 y_k)(t) \text{ for } k\in\{1,2,...,K\}.$$
Separability along with marginal stationarity ensures that
$$\text{Cov}\left(\cfrac{\partial y_k(t)}{\partial t_i},\cfrac{\partial y_k(t)}{\partial t_j}\right)=-\text{Cov}\left(y_k(t),\cfrac{\partial^2 y_k(t)}{\partial t_i\partial t_j}\right)=\lambda_k \delta_{i,j} I \text{ for } k\in\{1,2,...,K\}.$$
\end{remark}
\begin{remark}
We can choose to replace assumptions \ref{item:1} and \ref{item:3} with the stronger assumption of isotropic, centered and unit variance, independent Gaussian random fields. Isotropy would imply metric conformity, as desired in \ref{item:3}.
 \end{remark}

\begin{remark}
\label{scale}
We can in fact assume that in \ref{item:2}  that $$\text{Cov}({\grad y})={D \otimes \nu I} \text{ for some } \nu \in \real^{+},$$
with a scaled spatial covariance matrix $\nu I$.
In such a case, the spatial scale parameter $\nu$ shows up in our GKF in \ref{KFF-gen} as scaled LKCs with a scaling factor of $\nu^{j/2}$ for $\bm\calL_j$. We elaborate on this remark in Section \ref{PP} in Remark \ref{scale:exp}. In all our of our calculations, we assume $\nu=1$.
\end{remark}
Note that, we enforce distributional heterogenity amongst the building blocks $$\{y_i,\;\;i \in 1,2,..,K\},$$ by considering a KP separable covariance structure for the gradient fields $\grad y$, such that metrics induced by them are conformal in nature.
We emphasize that all the geometric calculations are with respect to the base spatial metric $g$.

\indent We need some assumptions on the function $F$, which when composed with the above non-identically distributed but independent Gaussian fields $ y$, yields heterogeneous Gaussian relatives $f$, our fields of interest.
\noindent We assume that $F\in \calC^2(\real^K)$ is real valued and satisfies for some $\epsilon>0$:
\begin{enumerate}[label=\textnormal{\textbf{(\arabic*)}},itemsep=2pt,parsep=2pt]
\item $\|\grad F\|$ is bounded on both sides on $F^{-1}(u-\epsilon, u+\epsilon)$ \label{itm:1}
\item $\grad F$ is Lipschitz on $F^{-1}(u-\epsilon,u+\epsilon)$ \label{itm:2}
\item Functions $\tilde{\calC}_F$, discussed in \eqref{func} are continuous in $(u-\epsilon, u+\epsilon)$ \label{itm:3}
\item $\lim_{\varepsilon\to 0}\cfrac{1}{2\varepsilon}\E\left[1_{\{|F(y)-u|<\varepsilon\}}|H_{n-1-l}(-\langle y,DF(y)\rangle) \|D^2\grad^2 F\|_{\otimes^{2l}\real^K}| \right]<\infty$, for all $n,l$. \label{itm:4}
\end{enumerate}

Finally, we assume that the domain set $\calK=F^{-1}[u,\infty)$ is smooth and convex, with shape operator of $\partial\calK$ bounded and critical radius of $\calK$ positive.

\begin{remark}
When the set $\calK$ is non-smooth, but locally convex, then similar calculations hold, though we will 
have to be careful about breaking up calculations on different pieces.
\end{remark}

\subsection{Main result}

With ${y}$, $F$ and domain set $\calK$ satisfying assumptions listed in Section \ref{assumption0} above, 
we are ready to state the GKF theorem for the heterogeneous Gaussian relatives with independent, 
marginally stationary components $$\{y_i, i=1,2,...,K\}$$ having separable structure for the gradient field 
and inducing conformal metrics. Even more strictly, we can consider isotropic Gaussian components with
separability for corresponding gradient fields. We outline the proof briefly in this section, proving the details 
of the results involved in the GKF in the following sections.

\begin{thm}
\emph{GKF generalized to heterogeneous Gaussian Related Fields:}
\label{KFF-gen} For $F\in \calC(\real^K)$ satisfying \ref{itm:1}, \ref{itm:2}, \ref{itm:3} and \ref{itm:4} in the assumptions for $F$, $\calK=F^{-1}[u,\infty)$ convex and smooth, and Gaussian random fields  $y=(y_1,..,y_K)$, satisfying assumptions \ref{item:1}, \ref{item:2},  \ref{item:3} and  \ref{item:4} on $M$, the kinematic formula for $f=F\circ y$ can be expressed as
\begin{equation}\label{eqn:main}
\E[\bm\chi(M\cap y^{-1}\calK)]=\sum_{j=0}^{n} (2\pi)^{-j/2}\bm\calL_j(M)\bm\calM_{j}^{\gamma_\real^K,D}(\calK),
\end{equation}
where the coefficients in the above expansion decouple into $\{\bm\calL_j(M)\}_{0\leq j\leq n}$  (LKCs), computed with respect to spatial metric $g$ and 
$\{\bm\calM_{j}^{\gamma_\real^K,D}(\calK)\}_{0\leq j\leq n}$ (GMFs), arising as coefficients in the Taylor series expansion of Gaussian volumes of ellipsoidal tubes as in \eqref{eq:gkfold}
\end{thm}

\begin{proof}
The expected Euler Characteristic is derived in terms of the LKCs on $M$ and the EC density functionals in Section \ref{PP} as
$$\mathbb{E}\left[\chi(M\cap y^{-1}\calK)\right]=\sum_{j=0}^{n}\bm{\mathcal{L}}_{j}(M){\rho}_{j}(F,u).$$
Independently, the Gaussian volume expansion of ellipsoidal tubes is computed as an expansion in GMFs, denoted as $\{\bm\calM_{j}^{\gamma_\real^K,D}(\calK)\}_{0\leq j\leq n}$ in Section \ref{GP}.
Having done the computations above, it is a matter of verifying that the coefficients $\bm\calM_{j}^{\gamma_\real^K,D}(\calK)$ in Theorem \ref{GVT}, followed with the derivation of EC densities $\tilde{\rho}_{n}(F,u)$ for fields $f$ in \ref{ECD} match up to constants. \\ Realizing that  $\partial\calK=F^{-1}{u}$ and $\eta_{y}$, the outward unit normal at $y\in\partial\calK$ in \ref{GVT}, identifies with $\grad F(y)$, we see that 
$$ \left(\dfrac{\lVert  D\eta_{y}\lVert}{\lVert  D^{1/2}\eta_{y} \lVert}\right)^{l-1-m} H_{l-1-m}\left(\cfrac{\langle y,D\eta_{y}\rangle}{\lVert  D\grad \eta_y\lVert}\right)$$
in \eqref{Gaussian Volume1}, coming from a Taylor series expansion of the standard Gaussian density function matches with
$$\left(\dfrac{\lVert  D\grad F(y)\lVert}{\lVert  D^{1/2}\grad F(y) \lVert}\right)^{l-1-m}H_{l-1-m}\left(\dfrac{\langle D\grad F(y),y \rangle}{ \lVert  D\grad F(y)\lVert}\right),$$
in \eqref{EC density expression}.\\
The other term in the $\epsilon$ neighborhood expansion of the ellipsoidal tube, denoted as 
$$\calM_{m+1}^{*}(\calK,dx),$$ derived from the Jacobian of an appropriate transformation, parametrizing the boundary of the ellipsoidal tube, in \ref{Jacobian}, 
matches with $$Tr^{\grad F^{\perp}}(D\grad^2 F_{\vert \grad F^{\perp}}/\lVert  D^{1/2}\grad F(y)\lVert)^{m} \calH_{K-1}(dx)$$ 
in \eqref{EC density expression}.
The two computations in \eqref{Gaussian Volume1} and \eqref{EC density expression} help us conclude that the EC densities indeed, agree up to constants with the coefficients of volumes of ellipsoidal tubes $T^{D}(\calK,\ep)$, yielding
$$\tilde{\rho}_{j}(F,u)=\cfrac{1}{(2\pi)^{j/2}}\bm\calM_{j}^{\gamma_\real^K,D}(\calK),$$
and thus, follows the Kinematic formula.
\end{proof}

As we go through the subsequent sections, all the claims made in the proof above become clear.

\section{Gaussian volumes of ellipsoidal tubes}
\label{GP}
\indent  In this section, we devote our attention to tubular expansions of certain geometric objects, which we refer to as ``ellipsoidal tubes", computed with respect to the standard Gaussian measure. Computations for volumes of regular tube neighborhoods around $M$ can be seen as early as in \cite{weyl1939volume} for an embedded manifold $M\subset \real^n$ and are credited to Steiner for compact, convex $M$. These volume expansions have also found place in \cite{gray1990tubes, schneider2013convex} later. For us, the geometric objects whose expansions lead to a GKF for heterogeneous, Gaussian related fields are no longer the regular tubes. The non-homogeneity in distribution of component Gaussian fields leads to a different geometry, which we call ellipsoidal tubes and define them below.

\indent In particular, we compute the Gaussian volume of an ellipsoidal tube around $\calK$ where $$\calK=F^{-1}[u,\infty)\in \real^K,$$ is assumed to be smooth and convex.

\indent For a positive definite $\Theta\in \real^{K \times K} $, 
we begin with the observation that the boundary of the set $\calK\oplus B_{\Qinv,\real^K}(\ep)$ for a convex, smooth set  ${\calK}$ can be parametrized by the map
$$(x, \eta_x) \to x+\ep \dfrac{\Qinv\eta_{x}}{\lVert \eta_{x}\lVert_ \Qinv},$$ 
where $\eta_x$ is the outward unit normal at $x$ on $\partial \calK$, and a $\Qinv$-norm is defined as:
$$\|x\|^2_{\Qinv}=x^T{\Qinv}x.$$


\begin{lemma}
\label{param}
For ${\calK}$ convex and ${\Theta} \in \real^{K \times K} > 0$, the map
$$ S(N({\calK})) \ni (z, \eta_z) \mapsto z + \epsilon \frac{\Qinv\eta_z}{\|\eta_z\|_{\Qinv}} $$
parameterizes $\partial ({\calK} \oplus B_{\Qinv,\real^K}(\epsilon))$, where
$$\|z-y\|^2_{\Qinv^{-1}}=(z-y)^T{\Theta^{-1}}(z-y).$$  
\end{lemma}

\begin{proof}
For $y \in \real^K$, consider the problem
\begin{equation}
\minimize_{z \in {\calK}}  \|z-y\|^2_{\Qinv^{-1}}
\end{equation}
where
$$
\|z-y\|^2_{\Qinv^{-1}}=(z-y)^T{\Qinv^{-1}}(z-y)
$$
as defined in the statement of the Lemma.
For $y \in {\calK}^c$, solving the above problem yields a pair
$$
\left(\hat{z}(y), y - \hat{z}(y)\right) \in N(\partial {\calK})
$$
satisfying the KKT conditions
$$
{\Qinv^{-1}}(\hat{z}(y) - y) = -c(y) \cdot (y - \hat{z}(y))
$$
for $c(y) \geq 0$ and $y - \hat{z}(y)$ is in the normal cone to ${\calK}$ at $\hat{z}(y)$.

\noindent Hence, solving this problem determines a map $h:{\calK}^c \rightarrow N(\partial {\calK})$. This map has
an inverse $\bar{h}$ defined
$$
(z, \eta_z) \mapsto z + \Qinv\eta_z
$$
with $\Qinv=\Q^{-1}$.
That is, 
$$
{\Qinv^{-1}}(z - \bar{h}(z, \eta_z)) = \eta_z,
$$
and hence
$$
z = \argmin_{w \in {\calK}}  \|w - \bar{h}(z, \eta_z)\|^2_{\Qinv^{-1}}.
$$

\noindent (At the risk of being a little pedantic we are identifying the matrix $\Qinv$ with a linear mapping $T_y\real^K \rightarrow T_y\real^K$ assuming that the basis
of this mapping in the standard basis is $\Qinv$. It therefore makes sense to write $\Qinv E$ and $\Qinv^{-1} E$ for vector fields on $\real^K$ or $\partial {\calK}$.)\\

\noindent Each $(z,\eta_z)$  has a value associated to it: $V_{\calK}(\bar{h}(z,\eta_z))$ where
\begin{equation} \label{eqn:V-K}
V_{\calK}(y) = \inf_{z \in {\calK}}  \|z-y\|^2_{\Qinv^{-1}}.
\end{equation}
We note here that
$$
\left\{y: V_{\calK}(y) \leq \epsilon^2 \right\} = {\calK} \oplus B_{\Qinv,\real^K}(\epsilon).
$$

\noindent We parameterize the set
$$
\left\{y: V_{\calK}(y)=\epsilon^2 \right\} = \partial ({\calK} \oplus B_{\Qinv,\real^K}(\epsilon))
$$
by constructing a map from $S(N({\calK}))$, the unit normal vectors of ${\calK}$ to this set.
Clearly,  
$$
V_{\calK}(\bar{h}(z, c \eta_z)) = c^2 \|\eta_z\|_{\Qinv}^2 \qquad \forall (z, \eta_z) \in S(N({\calK})), c > 0
$$
hence choosing 
$$
c(z, \eta_z) = \frac{\epsilon }{\|\eta_z\|_{\Qinv}}
$$
yields the desired parameterization.
\end{proof}

\begin{corollary}
\label{param_D}
Taking $\Qinv =D= \text{diag}(\lambda_1, \dots, \lambda_K)$ yields the map which parametrizes the boundary of $T^{D}(\calK,\ep)=\calK\oplus  B_{D,\real^K}(\ep) $ where $$\calK=F^{-1}[u,\infty)\in \real^K,$$ which we assume to be convex and smooth and $ B_{D,\real^K}(\ep)$ is the ellipsoid given by the set $$\{x\in \real^K: x^{T}D^{-1}x\leq \ep^2\}.$$  
\end{corollary}

\indent With the above parametrization of the surface of ${\calK} \oplus B_{D,\real^K}(\epsilon)$, the next lemma evaluates the Jacobian of the transformation 
$$(x, \eta_x)\to x+\ep \dfrac{D\eta_{x}}{\|\eta_{x}\|_{D}},$$
used to project a small patch $\calA_{\epsilon}$ on $\partial({{\calK} \oplus B_{D,\real^K}(\epsilon)})$ onto patch $\calA$ on $\partial{\calK}$.
\begin{Lemma}
\label{Jacobian}
Again, under the assumption of $\calK$ being smooth,  ${\Qinv} \in \real^{K \times K} > 0$  and  $$\eta_z = -\nabla F(z) / \|\nabla F(z)\|_2,$$ taken as the unique
outward pointing normal vector field which can be extended to a smooth vector field on a neighborhood of any patch on $\partial {\calK}$, the change of basis matrix for the transformation
$$(z, \eta_z) \mapsto z + \epsilon \frac{\Qinv\eta_z}{\|\eta_z\|_{\Qinv}}$$
is given by 
$$
\text{det}(I_{K-1} + \epsilon \cdot A(z)) = \sum_{j=0}^{K-1} \epsilon^j \text{detr}_j(A(z)) ,
$$
where,
$$
\begin{aligned}
A(z)_{ij}  &= \frac{1}{\|\nabla F(z)\|_{\Qinv}} \langle \nabla_{E_i} \Theta\nabla F, E_j \rangle_{\Qinv^{-1}} \biggl|_z.  \\
\end{aligned}
$$
\end{Lemma}

\begin{proof}
Let's look at the derivative of a linear projection of our map. Ignoring the term
$z$ define
$$
\param(z) = \frac{ \epsilon \Qinv \eta_z}{\|\eta_z\|_{\Qinv}}.
$$

\noindent A straightforward calculation shows that for any vector field $X$ on $\real^K$
\begin{equation}
\label{eq:deriv}
\begin{aligned}
\langle \nabla_X\param, v \rangle_I &=  \epsilon \left \langle \frac{\Qinv^{1/2} \nabla_X\eta}{\|\eta\|_{\Qinv}}, \left(I - \frac{\Qinv^{1/2}\eta\eta^T\Qinv^{1/2}}{\|\eta\|_{\Qinv}^2}\right) \Qinv^{1/2} v  \right \rangle_I \\
 &=  \epsilon \left \langle \frac{ \nabla_X\eta}{\|\eta\|_{\Qinv}}, \left(\Qinv -  \frac{\Qinv\eta\eta^T\Qinv}{\|\eta\|_{\Qinv}^2}\right)  v  \right \rangle_I. \\
\end{aligned}
\end{equation}
with  $\langle, \rangle_I$ the usual Euclidean inner product, $\Qinv^{1/2}$ the symmetric square root of $\Qinv$ (we could take non-symmetric square-roots if we had to, but it doesn't matter -- everything below
makes sense without square roots),
and $\nabla$ the usual Euclidean Levi-Civita connection (i.e. standard differentiation of vector fields).
\noindent Letting $\Q=\Qinv^{-1}$, define
$$
\langle X, Y \rangle_Q = \langle X, QY \rangle_I.
$$
Now, choose a frame $\{E_1,\dots, E_K\}$ such that
\begin{equation}\label{eqn:ONB}
\langle E_i, E_j \rangle_Q = \delta_{ij}, \qquad 1 \leq i,j \leq K
\end{equation}
with
$$
E_K = \frac{\Qinv \eta}{\|\eta\|_{\Qinv}}.
$$
Noting that $\param(z) = \epsilon E_K(z)$ it suffices to differentiate
$E_K$ to compute the change of measure term.

\noindent Note that, for $1 \leq i \leq K-1, 1 \leq j \leq K$:
$$
\begin{aligned}
0 &= E_i \langle E_K, E_j \rangle_{\Q} \\
&= E_i \langle E_K, {\Q}E_j \rangle_I \\
&= \langle E_K, {\Q}\nabla_{E_i} E_j \rangle_I + \langle \nabla_{E_i}E_K, \Q E_j \rangle_I\\
&= \frac{1}{\|\eta\|_{\Qinv}}\langle \eta, \nabla_{E_i} E_j \rangle_I + \langle \nabla_{E_i}E_K, E_j \rangle_{\Q}\\
&= -\frac{1}{\|\eta\|_{\Qinv}}\langle \nabla_{E_i}\eta,  E_j \rangle_I + \langle \nabla_{E_i}E_K, E_j \rangle_{\Q}.\\
\end{aligned}
$$
The last display uses the assumption that $\langle E_i, \eta \rangle_I = 0$ for $ 1 \leq i \leq K-1$ which is implied 
by our choice of frame $\{E_1, \dots, E_K\}$.

\noindent Our calculation \eqref{eq:deriv} shows that
$$
\begin{aligned}
\langle \nabla_{E_i}E_K, E_K \rangle_{\Q} &= \langle \nabla_{E_i}E_K, \frac{\eta}{\|\eta\|_{\Qinv}} \rangle_{I} \\
&= \frac{1}{\|\eta\|_{\Qinv}} \langle \nabla_{E_i}E_K, \eta \rangle \\
&= 0.
\end{aligned}
$$
Therefore,
$$
\begin{aligned}
\nabla_{E_i}E_K &= \sum_{j=1}^{K} \langle \nabla_{E_i}E_K, E_j \rangle_{\Q} E_j \\
&= - \sum_{j=1}^{K-1} \langle\frac{\eta}{\|\eta\|_{\Qinv}},  \nabla_{E_i}E_j \rangle_{I} E_j. \\
&= \frac{1}{\|\eta\|_{\Qinv}} \sum_{j=1}^{K-1} \langle \nabla_{E_i}\eta, E_j \rangle_{I} E_j. \\
\end{aligned}
$$

\noindent In terms of our frame, this means the change of basis matrix for the transformation is therefore
$$
\text{det}(I_{K-1} + \epsilon \cdot A(z))
$$
where
$$
\begin{aligned}
A(z)_{ij} &= \frac{1}{\|\eta_z\|_{\Qinv}} \langle \nabla_{E_i}\eta, E_j \rangle_{I} \biggl|_z  \\
 &= \frac{1}{\|\eta_z\|_{\Qinv}} \langle \nabla_{E_i}\Theta\eta, E_j \rangle_{\Q} \biggl|_z  \\
 &= \frac{1}{\|\nabla F(z)\|_{\Qinv}} \langle \nabla_{E_i} \Theta\nabla F, E_j \rangle_{\Q} \biggl|_z  \\
\end{aligned}
$$

\noindent Finally, the determinant is a polynomial in $\epsilon$ 
$$
\text{det}(I_{K-1} + \epsilon \cdot A(z)) = \sum_{j=0}^{K-1} \epsilon^j \text{detr}_j(A(z)) .
$$
\end{proof}

\indent To calculate the volume of  $T^{D}(\calK,\ep)$, we compute the surface area of an infinitesimally small patch $\calA_\ep$ on the surface of $T^{D}(\calK,\ep)$ by using the map in \ref{param} and the Jacobian in \ref{Jacobian} to project back to $\partial\calK$. Integrating over $\partial\calK$ and over $[0,\ep]$, we get an expansion for the volume for the ellipsoidal tube. Based on \ref{param} and \ref{Jacobian}, the next corollary gives the surface measure of a small patch $\calA_{\epsilon}$ on the surface of the ellipsoidal tube $\partial({\calK\oplus  B_{\Qinv,\real^K}(\ep)})$.

\begin{corollary}
\label{surface measure}
The surface measure of a patch  $\calA_{\epsilon}$ on the surface of $\partial{(\calK\oplus  B_{\Qinv,\real^K}(\ep))}$ is given by
\begin{eqnarray}
\label{Volume}
&& \calH_{K-1}(\calA_{\ep})=\sum_{j=1}^{K}\dfrac{\ep^{j-1}}{(j-1)!}\calM_{j}^{*}(\calK,\calA)
\end{eqnarray}
where
$$\calM_{j}^{*}(\calK,\calA)=(j-1)!\int_{\calA}detr_{j-1}(A(z)) \calH_{K-1}(dx).$$
\end{corollary}

\begin{proof}
Letting $\eta_x$ be the unique outward pointing normal vector at $x$ on $\partial \calK$, observe that for a patch $\calA$ on $\partial\calK$
\begin{eqnarray}
\label{Volume}
&& \calH_{K-1}(\calA_{\ep})=\calH_{K-1}\left(\left\{x+\ep \dfrac{\Qinv\eta_{x}}{\|\eta_{x}\|_{\Qinv}}: x\in \calA\right\}\right)\nonumber\\
&=& \int_{\calA}det(I_{K-1} + \epsilon \cdot A(z)) \calH_{K-1}(dx)\nonumber\\
&=&  \sum_{j=1}^{K}\ep^{j-1}\int_{\calA}detr_{j-1}(A(z)) \calH_{K-1}(dx)\nonumber\\
&=&  \sum_{j=1}^{K}\dfrac{\ep^{j-1}}{(j-1)!}\calM_{j}^{*}(\calK,\calA)\nonumber
\end{eqnarray}
with $$\calM_{j}^{*}(\calK,\calA)=(j-1)!\int_{\calA}detr_{j-1}(A(z)) \calH_{K-1}(dx).$$
\noindent The first equality follows from \ref{param}, while the third one follows using the change of basis matrix transformation derived in \ref{Jacobian}. 
\end{proof}

\indent We finally use the Taylor Series expansion of the integral of function $$\vp=\dfrac{1}{(2\pi)^{K/2}}e^{-\lVert x\lVert^2}$$ over $T^{D}(\calK,\ep)$ and \ref{surface measure} to get an expansion of Gaussian volume of $T^{D}(\calK,\ep)=\calK\oplus  B_{D,\real^K}(\ep) $, which gives the main theorem in this section. 
\begin{thm}
\label{GVT}
Under the condition 
\begin{equation}
\label{assumption}
\int_{\partial{\calK}} \dfrac{1}{1+\|z\|^\beta}\calM_j^{*}(\calK,dz) \text{ being bounded for all } 0\leq j \leq K, \text{ for some } \beta>0 ,
\end{equation}
the Gaussian volume of $T^{\Theta}(\calK,\ep)$, denoted by $\gamma^{\real^K}(T^{\Theta}(\calK,\ep))$,
can be represented as the following expansion
 \begin{eqnarray}
\label{Gaussian Volume1}
&& \gamma^{\real^K}(T^{\Qinv}(\calK,\ep))= \bm\calM_{0}^{\gamma_{\real^K},D}(\calK)+\sum_{l=1}^{n+K}\dfrac{\ep^l}{l!}\bm\calM_{l}^{\gamma_{\real^K},D}(\calK)+\calR(\calK),
\end{eqnarray}
with the remainder term in the expansion $\calR(\calK)$ is bounded above as
$$\calR(\calK)\leq \dfrac{\ep^{n+2}}{(n+2)!}C(\calK),$$
$C(\calK)$ being a constant. The coefficients in the expansion are given by
\begin{equation}\label{eqn:M0}
\bm\calM_{0}^{\gamma_{\real^K},D}(\calK)=\int_{\calK}\vp(x) d\lambda_{\real^K},
\end{equation}
and for $l\geq 1$,
\begin{eqnarray}
\label{coeff}
&&\bm\calM_{l}^{\gamma_{\real^K},D}(\calK)=\sum_{m=0}^{l-1}\dbinom{l-1}{m} (-1)^{l-1-m}\int_{\partial\calK}\cfrac{\|\Theta^{1/2}\eta_{x}\|}{\|\eta_{x}\|}\times \left( \dfrac{\|\Theta\eta_{x}\|}{\lVert \Theta^{1/2}\eta_{x}\lVert}\right)^{l-1-m} \\ \nonumber
&& \;\;\;\;\;\;\;\;\;\;\;\;\;\;\;\;\;\;\;\; \times H_{l-1-m}\left(\left\langle x,\dfrac{\Qinv\eta_{x}}{\|\Qinv\eta_x\|}\right\rangle\right)\vp(x)\calM_{m+1}^{*}(\calK,dx)
\end{eqnarray}
with 
$\calM_{j}^{*}(\calK,dx)$ defined in \ref{surface measure} and $\calM_{j}^{*}(\calK,.)=0$ for $j>K$.
\end{thm}

\begin{proof}
 The integral of $\vp$ over $T^{\Qinv}(\calK,\ep)$ can be obtained by an application of co-area formula (see equation (7.4.14) of
\cite{RFG}). Notice that $$T^{\Qinv}(\calK,\ep) = \cup_{0\le \delta\le \ep} \partial T^{\Qinv}(\calK,\delta),$$ where the foliations $\partial T^{\Qinv}(\calK,\delta)$
can be considered as level sets of the distance function $V_{\calK}$, defined in equation \eqref{eqn:V-K}, which corresponds to geodesic lengths w.r.t. the Riemannian metric given by
$$\widetilde{g_{ij}} = \lambda^{-1}_i\,\delta_{ij}.$$
Using Lemma L4 of \cite{BGK} we note that
\begin{equation}\label{eqn:grad-dist}\grad V_{\calK} = \frac{\Qinv^{-1}z}{\|\Qinv^{-1/2}z\|},\end{equation}
where $z$ is the vector connecting two points between which the distance is measured.
In this setting, as pointed in Lemma \ref{param}, \begin{equation}\label{eqn:connect}z = \frac{\Theta\eta_x}{\|\Theta^{1/2}\eta_x\|}.\end{equation}
where $$x=\arg\min_{z\in \calK}\|z-y\|^2,$$ for $y$ on $\partial {T^{\Qinv}(\calK,\delta)}$.
Collating equations \eqref{eqn:grad-dist} and \eqref{eqn:connect}, and using
equation (7.4.14) of \cite{RFG}, we obtain the integral of $\vp$ over $T^{\Qinv}(\calK,\ep)$ given by
$$\int_{T^{\Qinv}(\calK,\ep)}\vp(y)dy=\int_{\calK}\vp(y) d\lambda_{\real^K}+\int_{0}^{\epsilon}\int_{\partial\calK}\int_{A_{\delta}} \frac{\|\Qinv^{1/2}\eta_{w(y)}\|}{\|\eta_{w(y)}\|} \vp(y)d\calH_{K-1}(y)d\ep,$$
where $$w(y)=\arg\min_{z\in \calK}\|z-y\|^2.$$

We begin our calculation by computing the surface measure of an infinitesimally small patch $\calA_\delta$ on $\partial T^{D}(\calK,\delta)$. 
Using the parametrization in \ref{param} 
$$ S(N({\calK})) \ni (x, \eta_x) \mapsto x + \delta \param(x), \text{ where } \param(x)=\dfrac{\Qinv\eta_{x}}{\lVert \Qinv^{1/2}\eta_{x}\lVert}.$$
Note that for $y= x + \delta \param(x)$,  we have $w(y)=x$.
We invoke a change of variable argument, followed with a Taylor series expansion to obtain:
\begin{eqnarray}\label{surface area of patch}
&& \int_{\calA_\delta}\frac{\|\Qinv^{1/2}\eta_{w(y)}\|}{\|\eta_{w(y)}\|} \vp(y)d\calH_{K-1}(y)\nonumber\\
&=&  \sum_{t=0}^{K-1}\int_{\calA}\frac{\|\Qinv^{1/2}\eta_x\|}{\|\eta_x\|} \vp(x+\delta\param(x))\dfrac{\delta^{l}}{l!}\calM_{l+1}^{*}(\calK,dx)\nonumber\\
&=&   \sum_{t=0}^{K-1}\dfrac{\delta^{l}}{l!}\int_{\calA}\frac{\|\Qinv^{1/2}\eta_x\|}{\|\eta_x\|} \left(\sum_{j=0}^{n}\dfrac{\delta^{j}}{j!}\dfrac{\partial^{j}\vp}{\partial \param^j}\Bigg\vert_{x}
+\dfrac{\delta^{n+1}}{(n+1)!}\dfrac{\partial^{n+1}\vp}{\partial \param^{n+1}}\Bigg\vert_{\alpha(\delta,x)}\right)\calM_{l+1}^{*}(\calK,dx)\nonumber\\
&=& \sum_{l=0}^{n+K-1}\dfrac{\delta^l}{l!}\sum_{j=0}^{K}\dbinom{l}{j}\int_{\calA} \frac{\|\Qinv^{1/2}\eta_x\|}{\|\eta_x\|} (-1)^{j}\lVert \param(x)\lVert^{j} H_{j}(\langle x,\param(x)/\lVert \param(x)\lVert \rangle)\vp(x)\calM_{l-j+1}^{*}
(\calK,dx)\nonumber\\
&&+\sum_{l=0}^{K-1}\delta^l\int_{\calA}\frac{\|\Qinv^{1/2}\eta_x\|}{\|\eta_x\|} \dfrac{\delta^{n+1}}{(n+1)!}\dfrac{\partial^{n+1}\vp}{\partial \param^{n+1}}\Bigg\vert_{\alpha(\delta,x)}\calM_{l+1}^{*}(\calK,dx)\nonumber
\end{eqnarray}
We see that the remainder in the expansion given by
$$\calR(\calK)=\sum_{l=0}^{K-1}\delta^l\int_{\calA}\frac{\|\Qinv^{1/2}\eta_x\|}{\|\eta_x\|} \dfrac{\delta^{n+1}}{(n+1)!}\dfrac{\partial^{n+1}\vp}{\partial \param^{n+1}}\Bigg\vert_{\alpha(\delta,x)}\calM_{l+1}^{*}(\calK,dx)$$ is indeed bounded above by a constant times
$$\dfrac{\delta^{n+1}}{(n+1)!}\lambda_{\Qinv}^{\text{max}}\max_{1\leq i\leq K}\sup_{z}(1+\|z\|^\beta) \dfrac{\partial^{n+1}\vp}{\partial z_i^{n+1}}\Bigg\vert_{z} \sum_{l=0}^{K-1}\delta^{l}\int_{\calA}\dfrac{1}{1+\|\alpha(\delta,x)\|^\beta}\calM_{l+1}^{*}(\calK,dx),$$
where $\lambda_{\Qinv}^{\text{max}}$ is the largest eigen value of $\Theta$. Under our assumption \eqref{assumption} and the observation that $\vp(x)$ has derivatives decaying to $0$ faster than inverse powers of $x$, the remainder can be bounded above by constant $C(\calK)$.\\

Finally, ignoring remainder and integrating over $\partial \calK$ and over $[0,\ep]$, the Gaussian volume of $T^{D}(\calK,\ep)$ can be approximated as
\begin{align*}\label{volume}
&\int_{\calK}\vp d\lambda_{\real^K}+\sum_{l=1}^{n+K}\dfrac{\ep^l}{l!}\sum_{m=0}^{l-1}\dbinom{l-1}{m}(-1)^{l-1-m}\int_{\partial\calK}\cfrac{\|\Theta^{1/2}\eta_{x}\|}{\|\eta(x)\|}\lVert \param(x)\lVert^{l-1-m}\\
& \;\;\;\;\;\;\;\;\times H_{l-1-m}\left(\left\langle x,\dfrac{\Qinv\eta_{x}}{\|\Qinv\eta_x\|}\right\rangle\right)\vp(x)\calM_{m+1}^{*}(\calK,dx)\nonumber
\end{align*}
which proves
\eqref{Gaussian Volume1} with coefficients given by \eqref{coeff}.
\end{proof}

\section{Expected Euler characteristic}
\label{PP}
\indent In this section, we compute the expected Euler characteristic
$$\E[\bm\chi(M\cap f^{-1}[u,\infty)],$$ 
when $y=(y_1,...,y_K): M\to \real^{K}$, where $y_1,..,y_K$ satisfy regularity and marginal stationarity with the corresponding gradient field satisfying separability and metric conformity and $M$ is a $C^3$ manifold. Similar calculations for the expected Euler characteristic can be found in \cite{Adler81, taylor2003euler, RFG} with motivation to approximate probabilities of exceeding high levels in \cite{adler2000excursion} and an explicit approximation in \cite{taylor2005validity}. Again, we show a reduction of computations to calculation of the Lipschitz curvatures $\{\bm\calL_j(M)\}_{0\leq j\leq n}$, with respect to base spatial metric $g$ 
and EC density functionals $\tilde{\rho}_{n}(F,u)$. This can be viewed as a de-coupling into geometric information about the underlying manifold, captured by LKCs, and information about the distribution of the field, captured by EC densities. \\
\indent We begin this section by stating two lemmas, that constitute the basic tools used in the calculation of the expected Euler characteristic and subsequently, in computation of the integral representation of EC densities in \ref{IEC}. The derivation of the averaged Euler characteristic in \cite{Taylor06} hinges on a clever conditioning on $(y,\grad y)$, that reduces the math to computing the conditional mean and variance of Gaussian random fields and calculating conditional expectation of double forms $(y^*\grad^2 F)^l$. We re-derive these calculations adapted to the heterogeneity in the distribution of the component Gaussian fields.
\noindent The following two lemmas, as emphasized, constitute the main tools for the computations to go through for the heterogeneous Gaussian related random fields. 
\begin{Lemma}
\emph{Conditional Mean and Variance}
\label{CMV}
Consider a heterogeneous Gaussian related field $f=F\circ y$ 
where $y=(y_1,...,y_K): M\to \real^{K}$ is such that $y_i$ for $i \in\{1,2,..,K\}$ are unit variance, 
centered, independent random fields on a n dimensional manifold $M$ with gradient field, $\grad y$ having separable 
covariance structure $D\otimes I$ and inducing conformal metrics, as stated in \ref{item:3} and \ref{item:4}. Then, the conditional mean and covariance of 
$\grad^2f$, conditioned on $(y,\grad y)$, are given by 
\begin{equation}
\label{eqn:final-mean}
\textbf{Mean:\;} \mu_{y,\grad y}=\E\left(\left.\grad^2 f \right\vert y,\grad y\right)
=y^*\grad^2F(y) -\langle D\grad F(y),y \rangle I_n
\end{equation}
where the $(i,j)$ th element of $(y^*\grad^2F(y))_{(i,j)}=\sum_{k=1}^{K}\sum_{k'=1}^{K}
\cfrac{\partial^2 F(y)}{\partial y_k\partial y_{k'}}\cfrac{\partial y_k}{\partial t_i}
	\cfrac{\partial y_{k'}}{\partial t_j}$.

\begin{align*}
\textbf{Variance:\;} \E\left(\left. \left(\grad^2f -  \mu_{y,\grad y}\right)^2\right\vert y,\grad y\right)&=
 - \langle D\grad F(y),D\grad F(y) \rangle I^2 \\
&\;\;\;\;-2\langle D^{1/2}\grad F(y),D^{1/2}\grad F(y) \rangle R \numberthis\label{eqn:final-covariance}
\end{align*}
where $R$ is the Riemannian curvature tensor w.r.t. metric $g$.
\end{Lemma}

\begin{proof}
Let $\left\{E_i\right\}_{i=1}^m$ be an orthonormal frame field on $M$, then
\begin{align*}
\grad^2 f(E_i,E_j) &= \grad^2 F(y_*E_i,y_*E_j) + \sum_{k=1}^K \grad^2 y_k(E_i,E_j) \, \partial_k F \\
&= \grad^2 F(y_*E_i,y_*E_j) + \langle \grad F(y), \grad^2 \underline{y}(E_i,E_j)\rangle
\end{align*}
where $\grad^2 \underline{y}(E_i,E_j)$ represents the vector $$\left(\grad^2 y_1(E_i,E_j),\ldots,\grad^2 y_K(E_i,E_j) \right).$$
{\bf Mean:}
\begin{equation}\label{eqn:cm}
\E\left( \left.\grad^2 f(E_i,E_j)\right| y,\grad y\right) = \grad^2 F(y_*E_i,y_*E_j) + \E\left( \left. \langle \grad F(y), \grad^2 \underline{y}(E_i,E_j)\rangle\right|y,\grad y \right)
\end{equation}

\noindent Denoting $\Sigma_{y,\grad y}$ as the covariance of $(y,\grad y)$, where $\Sigma_{y,\grad y}$ has a block diagonal
form given as
$$\Sigma_{y,\grad y} = \left[ \begin{array}{cccc} 
I_K & 0 & \ldots & 0 \\
0 & \lambda_1I_n & \ldots & 0 \\
0 & 0 & \ldots & 0 \\
0 & 0 & \ldots & \lambda_k I_n
\end{array}\right],$$
and observing that the covariance between $\grad^2 \underline{y}(E_i,E_j)$ and $(y,\grad y)$ can be expressed as
$$\Sigma_{\grad^2 \underline{y}(E_i,E_j),(y,\grad y)} = \left[ \begin{array}{cc}
\left. \begin{array}{cccc}
-\delta_{ij}\lambda_1 & 0 & \ldots & 0\\
0 & -\delta_{ij}\lambda_2 & \ldots & 0\\
\vdots & \vdots & \ldots & \vdots \\
0 & 0 & \ldots & -\delta_{ij}\lambda_K \end{array} \right| & \underset{K \times nK}{\underline{0}}
\end{array}\right] = \left[- \delta_{ij}D \,\,\,\,\,\,\,\,\, \underset{K \times nK}{\underline{0}}\right],$$
we compute the conditional mean of $\left.\grad^2 \underline{y}(E_i,E_j)\right|y,\grad y$.
Finally, let us denote by $\Sigma_{\grad^2 \underline{y}(E_i,E_j)}$ the covariance matrix of 
$\grad^2 \underline{y}(E_i,E_j)$. With this notation, the conditional mean 
$\E\left( \left.\grad^2 f(E_i,E_j)\right| y,\grad y\right)$ can be expressed as
\begin{eqnarray*}
&& \E(\grad^2 f(E_i,E_j)| y,\grad f)\\
&=& \grad^2 F(y_*E_i,y_*E_j) + \left(\grad F(y)\right)^T \Sigma_{\grad^2 \underline{y}(E_i,E_j),(y,\grad y)} \Sigma^{-1}_{y,\grad y}\left( \begin{array}{c}y \\ \grad y\end{array}\right) \\
&=& \grad^2 F(y_*E_i,y_*E_j) - \delta_{ij}\langle D^{1/2}\grad F(y), D^{1/2}y\rangle. \\
\end{eqnarray*}
\textbf{Variance:}
\begin{remark}
Recall that, for the case of a single Gaussian field with $R(k)$ as Riemannian curvature tensor under 
the induced metric \cite[Lemma 12.2.1]{RFG}%
\begin{equation*}
-2R(k)=\mathbb{E}\left[ \mathbb{(\nabla }^{2}y_{k})^{2}\right] 
\end{equation*}
which in coordinates can be expressed as
\begin{equation*}
R_{ijmn}(k)=\mathbb{E}\left[ \mathbb{\nabla }^{2}y_{k }(E_{i},E_{m})\mathbb{%
\nabla }^{2}y_{k}(E_{j},E_{n})-\mathbb{\nabla }^{2}y_{k }(E_{i},E_{n})%
\mathbb{\nabla }^{2}y_{k }(E_{j},E_{m})\right] 
\end{equation*}%
where,
\begin{equation*}
\mathbb{\nabla }^{2}y_{k }(E_{i},E_{m})=\partial _{im}^{2}y_{k }-(\nabla
_{\partial _{i}}\partial _{m})y_{k }\text{ .}
\end{equation*}%
\end{remark}

\begin{remark}
Note that in case of Gaussian fields with conformity induced by the gradient fields (or the stricter case of isotropic Gaussian random fields), the induced metric is just a constant
multiple of the Euclidean metric, where the constant is simply the second spectral moment
of the underlying Gaussian random field. In such a case,
writing $R(k)$ and $R$ as the Riemannian curvature w.r.t. the induced and the Euclidean
metrics respectively, we have
$$R(k)(X,Y,Z,W) = \lambda_{k} R(X,Y,Z,W)$$
\end{remark}

\noindent Noting that
$$\left[\grad^ 2 f - \E\left(\grad^ 2 f | y,\grad y\right)\right](E_i,E_j) = 
\langle \grad F(y),\grad^2\underline{y}(E_i,E_j)\rangle - \delta_{ij}\langle D^{1/2}\grad F(y), D^{1/2}y \rangle$$
and
\begin{eqnarray*}
&& \left[\grad^ 2 f - \E\left(\grad^ 2 f | y,\grad y\right)\right]^2(E_i,E_j,E_m,E_n)\\
&=& \left[\grad^ 2 f - \E\left(\grad^ 2 f | y,\grad y\right)\right](E_i,E_m)
\left[\grad^ 2 f - \E\left(\grad^ 2 f | y,\grad y\right)\right](E_j,E_n) \\
&& - \left[\grad^ 2 f - \E\left(\grad^ 2 f | y,\grad y\right)\right](E_i,E_n)
\left[\grad^ 2 f - \E\left(\grad^ 2 f | y,\grad y\right)\right](E_j,E_m)
\end{eqnarray*}
Next conditioning on $(y,\grad y)$ and taking expectation, it is not difficult to see that
\begin{eqnarray*}
&&\E\left(\left.\left[\grad^ 2 f - \E\left(\grad^ 2 f | y,\grad y\right)\right]^2\right| y,\grad y\right)(E_i,E_j,E_m,E_n)\\
&=& - \|D\grad F(y)\|^2 I^2(E_i,E_j,E_m,E_n) - 2\|D^{1/2}\grad F(y)\|^2 R(E_i,E_j,E_m,E_n)
\end{eqnarray*}

\end{proof}

\indent The next lemma is a computational tool to compute the conditional expectation 
$$\E[(y^*\grad^2F)^{l}\vert y,\grad f], $$
which we are left to compute after a tower argument with expectation and plugging in the conditional mean and variance, computed in \ref{CMV}.

\begin{Lemma}
\label{Trace}
Let $E$ be an orthonormal frame bundle on $T\real^K$, then 
$$ \E[(y^*\grad^2F)^{l}\vert y,\grad_E f]
= Tr^{\grad F^{\perp}}(D\grad^2 F_{\vert \grad F^{\perp}})^{l}I^{l}+Err_l$$ where $\grad F^{\perp}(y)$ denotes the 
vector space generated by the span of vectors in $T_y\real^K$ which are orthogonal to $\grad F(y)$,
and
$$Err_l = \text{O}\left( \|\grad_E f\|^2 \|(\grad^2 F)^l\|_{\otimes^{2l}\real^K}\right)$$
where the big-O notation has the same interpretation as in \cite{Taylor06}. 
\end{Lemma}

\begin{proof} Notice that
the conditional expectation in above lemma has the same form as that in Corollary 2.3 of \cite{Taylor06}.
However, the inner product on $T_y\real^K$ needs to be defined appropriately so as to match the statement
of the aforementioned corollary from \cite{Taylor06}. More precisely,  suppose we take an orthonormal frame
$\{\bar{E}_{1,x}, \dots, \bar{E}_{d,x}\}$ in the parameter space $M$, we could rewrite the conditional expectation as
$$ \E[(y^*\grad^2F)^{l}\vert y,\grad_E f] = 
\E\left[ (y^*\grad^2F)^{l}\vert y, \langle y_* E_1,\grad F\rangle,\ldots,\langle y_* E_d,\grad F\rangle\right] $$
Now consider
$$
X_{i,y(x)} = y_*\left(\bar{E}_{i,x} \right) = \sum_{j=1}^K \langle\grad y_j,E_{i,x}\rangle \frac{\partial}{\partial y_j}\biggl|_{y(x)}.
$$
This is Gaussian on $T_{y(x)}\real^K$ and as $i$ ranges over $\{1,\dots,d\}$ we get IID copies, their distribution is 
$\gamma_V$ (conditional on $y(x)$). 
However to invoke Corollary 2.3 of \cite{Taylor06}, 
we shall rewrite
$$
X_{i,y(x)} = \sum_{j=1}^k \left(\lambda_j^{-1/2} \langle\grad y_j,E_{i,x}\rangle\right)\lambda_j^{1/2} \frac{\partial}{\partial y_j}\biggl|_{y(x)}
$$
where $\left(\lambda_j^{-1/2} \langle\grad y_j,E_{i,x}\rangle\right)$ are i.i.d. standard normal. This, in turn leads us to
define a new inner product $\langle\cdot,\cdot\rangle_{D^{-1}}$ on $T_{y(x)}\real^K$ as 
$$\langle V,W\rangle_{D^{-1}} = \sum_{i=1}^K \lambda^{-1}_{i}V_i(x)W_i(x)$$
ensuring that
$$
\left(\lambda_j^{1/2} \frac{\partial}{\partial y_j}\biggl|_{y(x)}\right), \qquad 1 \leq j \leq K
$$
is an orthonormal basis on $T_y\real^K$ w.r.t. the new inner product.
With this notation, we define $\grad F^{\perp}$ as the linear subspace of $T_y\real^K$ consisting of vectors 
orthogonal to $D\grad F$.



\end{proof}

\indent The next theorem computes the conditional expectation of double form $-(\grad^2 f)^n$, when conditioned on $y$ and $\grad f$. This leads to the derivation of the expected Euler Characteristic of excursion set, 
$$M\cap f^{-1}[u,\infty)=\{t\in M: f(t)\geq u\}.$$
\begin{thm}
\emph{Expected Euler Characteristic:}
\label{GGKF}
Let heterogeneous Gaussian related field, as considered in \ref{CMV}, be denoted by $f=F\circ y$, with component heterogeneous, but independent Gaussian units $$y=(y_1,...,y_K): M\to \real^{k}.$$
\begin{enumerate}[label=(\Alph*).]
\item For each $t\in M$ 
\begin{equation}
\label{series}
\frac1{n!}\E[(-\nabla ^{2}f)^{n}|y,\grad f](t)=\sum_{j=0}^{\lfloor \frac{n}2\rfloor}\dfrac{(-R_t)^j}{j!}\alpha_j(t)
\end{equation}
where $R$ is the Riemannian curvature tensor of the Manifold wrt to metric g and $\alpha_j(t)$ represent random double forms.
\item The expected Euler Characteristic can be represented as 
\begin{equation}
\label{series 1}
\mathbb{E}\bm\chi(M\cap f^{-1}([u,\infty))=\sum_{j=0}^{n}\bm{\mathcal{L}}_{j}(M){\rho}_{j}(F,u), 
\end{equation}
for functionals ${\rho}_{j}(F,u)$ representing the EC densities and $\bm{\mathcal{L}}_{j}(M)$ representing the LKC measures.
\end{enumerate}
\end{thm}

\begin{proof}
\begin{enumerate}[label=(\Alph*).]
\item In the proof, the expectation is evaluated at $t$, but, we choose to suppress the notation $t$ for convenience. To prove the first part of the theorem, we use the tower property of expectation to condition on the field $y$ and gradient $\grad y$. This reduces computations to the conditional mean and variances of a Gaussian random field, that is
$$\E\left[-(\nabla ^{2}f)^{n}\vert y, \grad f \right]=\E\left[\E[(-\nabla ^{2}f)^{n}|y,\grad y]\lvert y, \grad f \right],$$
We simplify the inner expectation using a binomial expansion of
$$\left(-(\grad^2 f-\mu_{y,\grad y})-\mu_{y,\grad y}\right)^n,$$ and by plugging in the conditional mean and conditional variance using \eqref{eqn:final-mean} and \eqref{eqn:final-covariance}. This follows from the observation that the expectation of a Gaussian double form simplifies into a binomial expansion of its mean and variance, expressed in the below remark.
\begin{remark}
If $Z$ is a Gaussian double form, then
$$\E(Z^k)=\sum_{i=0}^{\lfloor \frac{k}2\rfloor}\cfrac{k!}{(k-2i)! i! 2^i}\mu^{k-2i}\Sigma^{i},$$
where $$\mu=\E(Z) \text{ and } \Sigma=\E(Z-\E(Z))^2.$$
The reader can check \cite{taylor2003euler} for a derivation of the above.
\end{remark}
This gives us
\begin{eqnarray}
\label{innerexp}
&& \frac1{n!}\E\left( \left.\left(-\grad^2 f+\mu_{y,\grad y}-\mu_{y,\grad y}\right)^n\right\vert y,\grad y\right)  \nonumber\\
&=& \frac1{n!} (-1)^n \sum_{i=0}^{\lfloor \frac{n}2\rfloor} \frac{n!}{(n-2i)!i! 2^i} 
\left[ y^*\grad^2F(y) -\langle D\grad F(y),y \rangle I_n\right]^{n-2i}\nonumber\\
&& \times \left[- \langle D\grad F(y),D\grad F(y) \rangle I^2 
-2\langle D^{1/2}\grad F(y),D^{1/2}\grad F(y) \rangle R \right]^{i} \nonumber\\
&=& (-1)^n \frac1{n!}\sum_{i=0}^{\lfloor \frac{n}2\rfloor} \frac{n!}{(n-2i)!i! 2^i} 
\left[ y^*\grad^2F(y) -\langle D\grad F(y),y \rangle I_n\right]^{n-2i}\nonumber\\
&& \times \lVert  D^{1/2}\grad F(y)\lVert^{2i}\left[- \dfrac{\lVert  D\grad F(y)\lVert^2}{\lVert  D^{1/2}\grad F(y)\lVert^2} I^2 
-2 R \right]^{i} \nonumber\\
&=&  \sum_{j=0}^{\lfloor \frac{n}2\rfloor}\sum_{i=j}^{\lfloor \frac{n}2\rfloor}(-1)^{n+i} \frac{1}{(n-2i)!i! 2^i} \frac{i!}{(i-j)!j!}
\left[\dfrac{y^*\grad^2F(y)}{ \lVert  D^{1/2}\grad F(y)\lVert} -\dfrac{\langle D\grad F(y),y \rangle}{ \lVert  D^{1/2}\grad F(y)\lVert} I_n\right]^{n-2i}\nonumber\\
&&  \lVert  D^{1/2}\grad F(y)\lVert^{n}\left(\dfrac{\lVert  D\grad F(y)\lVert^2}{\lVert  D^{1/2}\grad F(y)\lVert^2} I^2  \right)^{i-j} 
(2 R)^j \nonumber\\
&=&  \sum_{j=0}^{\lfloor \frac{n}2\rfloor}\dfrac{(-R)^j}{j!} \lVert  D^{1/2}\grad F(y)\lVert^{n}\sum_{i=j}^{\lfloor \frac{n}2\rfloor}(-1)^{i}\frac{1}{(n-2i)!i! 2^i} \frac{i!}{(i-j)!j!}\left(\dfrac{\lVert  D\grad F(y)\lVert^2}{\lVert  D^{1/2}\grad F(y)\lVert^2} I^2  \right)^{i-j} \nonumber\\
&& \times\sum_{l=0}^{(n-2i)}\frac{(n-2i)!}{l!(n-2i-l)!}\left(y^*(-\grad^2F(y)/ \lVert  D^{1/2}\grad F(y)\lVert)\right)^{l}\left(\dfrac{\langle D\grad F(y),y \rangle}{ \lVert  D^{1/2}\grad F(y)\lVert} I_n\right)^{n-2i-l} \nonumber\\
&=&  \sum_{j=0}^{\lfloor \frac{n}2\rfloor}\dfrac{(-R)^j}{j!}\lVert  D^{1/2}\grad F(y)\lVert^{n}\sum_{l=0}^{n-2j}\dfrac{\left(y^*(-\grad^2F(y)/ \lVert  D^{1/2}\grad F(y)\lVert)\right)^{l}}{l!}\nonumber\\
&& \;\;\;\;\;\times\dfrac{(-1)^{(n-2j-l)}}{(n-2j-l)!} \left(\dfrac{\lVert  D\grad F(y)\lVert}{\lVert  D^{1/2}\grad F(y) \lVert}\right)^{n-2j-l} H_{n-2j-l}\left(\dfrac{\langle D\grad F(y),y \rangle}{ \lVert  D\grad F(y)\lVert}\right)I_{n}^{n-2j-l}
\end{eqnarray}

\noindent The third equality is obtained by a binomial expansion of 
$$\left[- \dfrac{\lVert  D\grad F(y)\lVert^2}{\lVert  D^{1/2}\grad F(y)\lVert^2} I^2 
-2 R \right]^{i},\;\;\text{for each } i\in \left\{0,1,...,\lfloor \frac{n}2\rfloor\right\},$$ and an interchange of summations.
The subsequent equality follows by another binomial expansion of the expression 
$$\left[\dfrac{y^*\grad^2F(y)}{ \lVert  D^{1/2}\grad F(y)\lVert} -\dfrac{\langle D\grad F(y),y \rangle}{ \lVert  D^{1/2}\grad F(y)\lVert} I_n\right]^{n-2i},\;\text{for each } i\in\left\{j,j+1,...,\lfloor\frac{n}2\rfloor\right\}.$$ 
The final equality follows by another interchange of summations and clubbing of terms to get Hermite polynomials as a function of  $\dfrac{\langle D\grad F(y),y \rangle}{ \lVert  D\grad F(y)\lVert}$.\\

\noindent With inner expectation evaluated in \eqref{innerexp}, we apply Lemma \ref{Trace} to evaluate $$ \E[(y^*\grad^2F)^{l}\vert y,\grad f],$$
which finally yields 
\begin{eqnarray}
\label{final exp}
&& \frac1{n!}\E[(-\nabla ^{2}f)^{n}|y,\grad f] \nonumber\\
&&= \sum_{j=0}^{\lfloor \frac{n}2\rfloor}\dfrac{(-R)^j}{j!}I_{n}^{n-2j}\lVert  D^{1/2}\grad F(y)\lVert^{n}\times  \sum_{l=0}^{n-2j} \dfrac{(-1)^{(n-2j-l)}}{(n-2j-l)!} \left(\dfrac{\lVert  D\grad F(y)\lVert}{\lVert  D^{1/2}\grad F(y) \lVert}\right)^{n-2j-l} \nonumber \\
&& \;\;\;\;\;\times H_{n-2j-l}\left(\dfrac{\langle D\grad F(y),y \rangle}{ \lVert  D\grad F(y)\lVert}\right)\times \;Tr^{\grad F^{\perp}}(D\grad^2 F_{\vert \grad F^{\perp}}/\lVert  D^{1/2}\grad F(y)\lVert)^{l}.
\end{eqnarray}
We ignore the error term in \ref{Trace}, that contributes
$$ \sum_{j=0}^{\lfloor \frac{n}2\rfloor}\sum_{l=0}^{n-2j} R^{j}I^{n-2j-l}H_{n-2j-l}\left(\dfrac{\langle D\grad F(y),y \rangle}{ \lVert  D\grad F(y)\lVert}\right)Err_l,$$
upto constants.
\noindent This proves \eqref{series}, with the random forms
\begin{eqnarray}
&& \alpha_j(t)=I_{n}^{n-2j}\lVert  D^{1/2}\grad F(y)\lVert^{n}\times  \sum_{l=0}^{n-2j} \dfrac{(-1)^{(n-2j-l)}}{(n-2j-l)!} \left(\dfrac{\lVert  D\grad F(y)\lVert}{\lVert  D^{1/2}\grad F(y) \lVert}\right)^{n-2j-l} \nonumber\\
&& \;\;\;\;\;\;\;\;\;\; \times \;\;H_{n-2j-l}\left(\dfrac{\langle D\grad F(y),y \rangle}{ \lVert  D\grad F(y)\lVert}\right)\times \;Tr^{\grad F^{\perp}}(D\grad^2 F_{\vert \grad F^{\perp}}/\lVert  D^{1/2}\grad F(y)\lVert)^{l}\nonumber
\end{eqnarray}
\item The second part of the theorem follows from the expectation metatheorem for counting critical points of $f$ in $M$ above level $u$ and index $k$ and Morse's representation for Euler characteristic, that leads to computation of its expected value. We thus have, 
\begin{equation}
\label{EC}
\mathbb{E}\bm\chi(M\cap f^{-1}([u,\infty))=\int_{M}  \E\left[ \left.\mathbb{I} (f\geq u)\text{det}(-\nabla ^{2}f)\right\vert \grad f =0\right]Vol_{M,g}.
\end{equation}
\noindent Using the definition of trace so that for any double form $A$, we have 
$$\text{det}(A)=\frac{1}{n!}Tr(A^n),$$
\begin{eqnarray}\label{eqn:EC-excursion}
&& \mathbb{E}\bm\chi(M\cap f^{-1}([u,\infty)) \nonumber\\
&=&\frac{1}{n!}\int_{M} \E\left[\lim\limits_{\epsilon\to 0} \left.\mathbb{I} (f\geq u,\|\grad f\|<\epsilon)Tr^{M}((-\nabla ^{2}f)^{n})\right\vert \grad f\right]Vol_{g}\nonumber\\
&=& \frac{1}{n!}\int_{M} \E\left( \lim\limits_{\epsilon\to 0} \left.\mathbb{I} (f\geq u,\|\grad f\|<\epsilon)\E\left[Tr^{M}((-\nabla ^{2}f)^{n})\right\vert f, \grad f \right]\right)Vol_{g}\nonumber\\
&=& \frac{1}{n!}\int_{M} \E\left(  \lim\limits_{\epsilon\to 0} \left.\mathbb{I} (f\geq u,\|\grad f\|<\epsilon)Tr^{M}\E\left[\left\{\E(-\nabla ^{2}f)^{n}|y,\grad y\right\}\right\vert y, \grad f\right]\right)Vol_{g}.\nonumber
\end{eqnarray}
\noindent Using \eqref{series}, we can conclude that $\mathbb{E}\chi(M\cap f^{-1}([u,\infty)) $ equals
\begin{equation}
\frac{1}{n!}\int_{M}\sum_{j=0}^{\lfloor \frac{n}2\rfloor} (2\pi)^{j}\E\left(  \lim\limits_{\epsilon\to 0} \mathbb{I} (f\geq u,\|\grad f\|<\epsilon) \alpha_j\right)\dfrac{Tr^{M}(-R)^j}{(2\pi)^{j}j!}Vol_{g}, 
\end{equation}
where $Tr^{M}(-R)^j$, the geometry from the Riemannian structure induced by $g$ contributes to the LKC and the expectation, computed w.r.t the standard Gaussian density on $\real^K$ form the EC functionals $\rho$ .
\end{enumerate}
\end{proof}

\begin{remark}
\label{scale:exp}
This is an elucidation on Remark \ref{scale} in \ref{intro}. If we have a scaled spatial covariance matrix by a factor of $\nu$, that is the induced metric
$$g_{i,j}^k=\lambda_k \times \nu\cdot g_{i,j},$$
and we carry out computations w.r.t. the canonical spatial metric $g$, then the scaled versions of conditional mean and variance are
\begin{equation*}
\label{eqn:mean:scaled}
 \mu_{y,\grad y}=\E\left(\left.\grad^2 f \right\vert y,\grad y\right)
=y^*\grad^2F(y) -\nu\langle D\grad F(y),y \rangle I_n
\end{equation*}
\begin{align*}
\E\left(\left. \left(\grad^2f -  \mu_{y,\grad y}\right)^2\right\vert y,\grad y\right)&=
 - \nu^2\langle D\grad F(y),D\grad F(y) \rangle I^2 \\
&\;\;\;\;-2\nu\langle D^{1/2}\grad F(y),D^{1/2}\grad F(y) \rangle R.
\end{align*}
Plugging the scaled versions in the calculations, and noting that $Tr^M$ scales as $\nu^{-n}$ and $Vol_{g^k}$ scales as $\nu^{n/2}$, we obtain the expected Euler characteristic as
$$\sum_{j=0}^{n}\nu^{j/2}\bm{\mathcal{L}}_{j}(M){\rho}_{j}(F,u)=\sum_{j=0}^{n}\bm{\mathcal{L}}^{\nu}_{j}(M){\rho}_{j}(F,u),$$
where $\bm{\mathcal{L}}^{\nu}$ is the LKC computed w.r.t to the induced spatial metric $\nu \cdot g$. 
\end{remark}

\begin{remark} 
To see the GKF in \cite{Taylor06} as a special case of \ref{GGKF}, we let $D=I$ in our computations in \ref{GGKF}. The difference is that our LKCs are computed w.r.t. the base spatial metric $g$, whereas the ones in \cite{Taylor06} are w.r.t the induced spatial metric $\nu\cdot g$. Finally, Remark \ref{scale:exp} shows that our calculations match as we switch to the induced LKCs.
\end{remark}

\section{Integral Representation of EC densities}
\label{IEC}
\indent In this section, we complete the details of proof of \ref{KFF-gen}, through an integral representation for the EC densities of heterogeneous Gaussian related fields with the Gaussian building blocks, $y$ on $\real^n$, satisfying \ref{item:1}, \ref{item:2}, \ref{item:3} and \ref{item:4}, stated in \ref{KFF}. The integral form of EC densities is seen to match with the coefficients in the volume expansion of the ellipsoidal tubes, introduced in \ref{GP}. \\
\indent We begin with a lemma, which evaluates a conditional expectation of functions of the random field restricted to the first $n-1$ coordinates, used in calculation of EC density, denoted as ${\rho}_n(F,u)=\rho_{f,n}(u)$. \\
\indent We introduce few notations for the section. Denote the gradient of $f$ with respect to the first $(n-1)$ coordinates only as
$$\grad f_{\vert (n-1)}=\left(\cfrac{\partial f}{\partial t_1},\cfrac{\partial f}{\partial t_2},...,\cfrac{\partial f}{\partial t_{n-1}}\right),$$ 

\noindent the joint density of $(f,\grad f_{\vert (n-1)})$ at $(u,0)$ as
$\phi_{f,\grad f_{\vert (n-1)}}(u,0),$ 
and the Hessian of $f$ restricted to again the first $(n-1)$ coordinates as
$\grad^2 f_{\vert (n-1)}.$
\begin{Lemma}
\label{in_exp}
With the same set up as \ref{GGKF} with a heterogeneous Gaussian related field $f$, we have 
\begin{eqnarray}
\label{inner expectation}
&& \calN_n(F,y)= \E\left[\left(\frac{\partial f}{\partial t_n}\right)^{+}\det(-\grad^2 f_{\vert (n-1)}) \;\Bigg\vert y,\frac{\partial f}{\partial t_i},\;1\leq i\leq (n-1)\right] \nonumber\\
&=& \dfrac{\lVert  D^{1/2}\grad F(y)\lVert^{n}}{(2 \pi)^{1/2}}\times  \sum_{m=0}^{n-1}{(-1)^{(n-1-m)}}\dbinom{n-1}{m}\left(\dfrac{\lVert  D\grad F(y)\lVert}{\lVert  D^{1/2}\grad F(y) \lVert}\right)^{n-1-m} \nonumber\\
&& H_{n-1-m}\left(\dfrac{\langle D\grad F(y),y \rangle}{ \lVert  D\grad F(y)\lVert}\right)  \times Tr^{\grad F^{\perp}}(D\grad^2 F_{\vert \grad F^{\perp}}/\lVert  D^{1/2}\grad F(y)\lVert^{m})+ Tr(\text{Err}^{n-1}),\nonumber
\end{eqnarray}
where $$\text{Err}^{n-1}=\sum_{m=0}^{n-1}I^{n-1-m}H_{n-1-m}({\langle D\grad F(y),y \rangle})O(\|\grad f_{\vert (n-1)}\|^2\|\grad^2 F^l\|_{\otimes^{2l}\real^K}).$$
\end{Lemma}
\begin{proof}
We can write
\begin{eqnarray}\label{eqn:ec}
&& \E\left[\left(\frac{\partial f}{\partial t_n}\right)^{+}det(-\grad^2 f_{\vert (n-1)}) \;\Bigg\vert y,\frac{\partial f}{\partial t_i},\;1\leq i\leq (n-1)\right]  \nonumber\\
&=& \E\left[\left(\frac{\partial f}{\partial t_n}\right)^{+}\;\Bigg\vert y,\frac{\partial f}{\partial t_i},\;1\leq i\leq (n-1)\right]\E\left[det(-\grad^2 f_{\vert (n-1)})\Bigg\vert y,\dfrac{\partial f}{\partial t_i},\;1\leq i\leq n-1\right] \nonumber \\
&=& \dfrac{\lVert  D^{1/2}\grad F(y)\lVert}{(2 \pi)^{1/2}} \times \lVert  D^{1/2}\grad F(y)\lVert^{n-1} \sum_{m=0}^{n-1}{(-1)^{(n-1-m)}}\dbinom{n-1}{m}\left(\dfrac{\lVert  D\grad F(y)\lVert}{\lVert  D^{1/2}\grad F(y) \lVert}\right)^{n-1-m} \nonumber\\
&&  \times H_{n-1-m}\left(\dfrac{\langle D\grad F(y),y \rangle}{ \lVert  D\grad F(y)\lVert}\right)\times Tr^{\grad F^{\perp}}(D\grad^2 F_{\vert \grad F^{\perp}}/\lVert  D^{1/2}\grad F(y)\lVert^{m})+ Tr(\text{Err}^{n-1}), \nonumber
\end{eqnarray}
where expectation $\E\left[det(-\grad^2 f_{\vert (n-1)})\Bigg\vert y,\dfrac{\partial f}{\partial t_i},\;1\leq i\leq n-1\right] $ is already evaluated in \eqref{final exp}.\\

\noindent The proof hinges on the observation that $${\partial f}/{\partial t_n} \perp (\left\{{\partial f}/{\partial t_i},\;1\leq i\leq n-1\right\}; \grad^2 f_{\vert (n-1)})\;\Big\lvert y.$$ 
\end{proof}

\indent We conclude the section with the derivation of the EC densities in an integral form, which matches with the coefficients in the Taylor expansion of ellipsoidal tubes.

\begin{thm}
\label{ECD}
Under the same set up as \ref{GGKF}, that is field $f=F\circ y$ with marginally stationary, zero-mean, unit variance, independent Gaussian fields with the gradient field having separable covariance structure and metric conformity, as considered in \ref{CMV}, the EC density functionals
$\{{\rho}_{n}(F,u), n\in \mathbb{Z}^+\}$, that appear in the series approximation of the expected Euler Characteristic in \eqref{series 1} can be expressed as
\begin{align*}
\label{EC density expression}
{\rho}_{n}(F,u) &= \dfrac{1}{(2\pi)^{n/2}}\sum_{m=0}^{n-1}(-1)^{n-1-m}\dbinom{n-1}{m}\int_{F^{-1}(z)}\frac{\|D^{1/2}\grad F(y)\|}{\|\grad F(y)\|}\left(\dfrac{\lVert  D\grad F(y)\lVert}{\lVert  D^{1/2}\grad F(y) \lVert}\right)^{n-1-m}\\ \nonumber
& \;\;\;\;\;\times H_{n-1-m}\left(\dfrac{\langle D\grad F(y),y \rangle}{ \lVert  D\grad F(y)\lVert}\right) Tr^{\grad F^{\perp}}(D\grad^2 F_{\vert \grad F^{\perp}}/\lVert  D^{1/2}\grad F(y)\lVert)^{m}  \\ \nonumber
&\;\;\;\;\; \times {(2\pi)^{-K/2}}e^{{-\lVert x\lVert^2}/2}d\calH_{K-1}(x). \numberthis
\end{align*}
\end{thm}
\vspace{0.4cm}
\begin{proof}
The EC densities are calculated as
\begin{eqnarray}
\label{integral representation EC}
&&  {\rho}_n (F,u)=\rho_{f,n} (u)\nonumber\\
&=& \E\left[\left(\frac{\partial f}{\partial t_n}\right)^{+}det(-\grad^2 f_{\vert (n-1)})\Big\lvert f=u,\grad f_{\vert (n-1)}=0\right]\phi_{f,\grad f_{\vert (n-1)}}(u,0) \nonumber\\
&=& \lim_{\epsilon\to 0}\dfrac{1}{2\epsilon \calV_{n-1}(\epsilon)}\E\left[\mathbb{I}{(\vert f-u\vert <\epsilon)} \mathbb{I}{(\lVert \grad f_{\vert (n-1)}\lVert <\epsilon)}\left(\frac{\partial f}{\partial t_n}\right)^{+}det(-\grad^2 f_{\vert (n-1)})\right] \nonumber\\
&=& \lim_{\epsilon\to 0}\dfrac{1}{2\epsilon^n \calV_{n-1}}\E\left[\mathbb{I}{(\vert f-u\vert <\epsilon)} \mathbb{I}{(\lVert \grad f_{\vert (n-1)}\lVert <\epsilon)}\calN_n(F,y)\right]\nonumber
\end{eqnarray}
with $$\calN_n(F,y)=\E\left[\left(\frac{\partial f}{\partial t_n}\right)^{+}det(-\grad^2 f_{\vert (n-1)}) \;\Bigg\vert y,\frac{\partial f}{\partial t_i},\;1\leq i\leq (n-1)\right]$$ calculated in \ref{in_exp}, and
 $$\calV_{n}(\epsilon)=\text{Vol} (\calB(\epsilon)) \text{ and } \calV_n=\text{Vol} (\calB(0,1)) \text{ with } \calB(0,\epsilon)\subset \real^{n}.$$
 
\noindent Noting that $$ \lim_{\epsilon\to 0}\dfrac{1}{2\epsilon^n \calV_{n-1}}\E\left[\mathbb{I}{(\vert f-u\vert <\epsilon)} \mathbb{I}{(\lVert \grad f_{\vert (n-1)}\lVert <\epsilon)}Tr(\text{Err}^{n-1})\right]=0,$$
and by defining
\begin{align*}
\label{inner expectation expression}
\calC_F(y)&=\left(\dfrac{\lVert  D\grad F(y)\lVert}{\lVert  D^{1/2}\grad F(y) \lVert}\right)^{n-1-m}H_{n-1-m}\left(\dfrac{\langle D\grad F(y),y \rangle}{ \lVert  D\grad F(y)\lVert}\right)\nonumber \\
& \;\;\;\;\;\;\;\;\times Tr^{\grad F^{\perp}}(D\grad^2 F_{\vert \grad F^{\perp}}/\lVert  D^{1/2}\grad F(y)\lVert)^{m}, \nonumber
\end{align*}
we have
\begin{eqnarray}
\label{final expression EC}
&& {\rho}_{n}(F,u)\nonumber\\
&=& \lim_{\epsilon\to 0}\dfrac{1}{2\epsilon^n\calV_{n-1}}\times \sum_{m=0}^{n-1}(-1)^{n-1-m}\dbinom{n-1}{m}\nonumber\\
&& \;\;\;\;\; \times\E\left[\E\left[\mathbb{I}{(\vert f-u\vert <\epsilon)} \mathbb{I}{(\lVert \grad f_{\vert (n-1)}\lVert <\epsilon)}\dfrac{\lVert  D^{1/2}\grad F(y)\lVert^{n}}{(2 \pi)^{1/2}}\calC_{F}(y)\Bigg\vert y\right]\right]\nonumber\\
&=& \lim_{\epsilon\to 0}\dfrac{1}{2\epsilon}\times \sum_{m=0}^{n-1}(-1)^{n-1-m}\dbinom{n-1}{m}\nonumber\\
&& \;\;\;\;\; \times\E\left[\dfrac{\gamma_{\real^{n-1}}(\B(0,\epsilon/\lVert  D^{1/2}\grad F(y)\lVert))}{\calV_{n-1}(\epsilon/\lVert  D^{1/2}\grad F(y)\lVert)^{n-1}}\mathbb{I}{(\vert F(y)-u\vert <\epsilon)} \dfrac{\lVert  D^{1/2}\grad F(y)\lVert}{(2 \pi)^{1/2}}\calC_{F}(y)\right]\nonumber\\
&=& \lim_{\epsilon\to 0}\dfrac{1}{2\epsilon}\times \sum_{m=0}^{n-1}(-1)^{n-1-m}\dbinom{n-1}{m}\dfrac{1}{(2\pi)^{n/2}}\nonumber\\
&& \;\;\;\;\; \times\E\left[\mathbb{I}{(\vert F(y)-u\vert <\epsilon)} \lVert  D^{1/2}\grad F(y)\lVert\calC_{F}(y)\right] \nonumber\\
&=& \sum_{m=0}^{n-1}(-1)^{n-1-m}\dbinom{n-1}{m}\dfrac{1}{(2\pi)^{n/2}} \lim_{\epsilon\to 0}\dfrac{1}{2\epsilon}\E\left[\mathbb{I}{(\vert F(y)-u\vert <\epsilon)} \lVert  D^{1/2}\grad F(y)\lVert\calC_{F}(y)\right] \nonumber\\
&=& \sum_{m=0}^{n-1}(-1)^{n-1-m}\dbinom{n-1}{m}\dfrac{1}{(2\pi)^{n/2}}\nonumber \\
&&  \;\;\;\;\;\times\lim_{\epsilon\to 0}\dfrac{1}{2\epsilon}\int_{(u-\epsilon,u+\epsilon)}\int_{F^{-1}(z)} \frac{\|D^{1/2}\grad F(y)\|}{\|F(y)\|}\calC_{F}(y)\dfrac{1}{(2\pi)^{K/2}}e^{{-\lVert y\lVert^2}/2}d\calH_{K-1}(y)dz  \nonumber \\
&=& \dfrac{1}{(2\pi)^{n/2}}\sum_{m=0}^{n-1}(-1)^{n-1-m}\dbinom{n-1}{m} \tilde{\calC}_{F}(u),\nonumber
\end{eqnarray}
where
\begin{equation}
\label{func}
 \tilde{\calC}_{F}(z)=\int_{F^{-1}(z)}\calC_{F}(y)\dfrac{1}{(2\pi)^{K/2}}e^{{-\lVert x\lVert^2}/2}d\calH_{K-1}(y).
 \end{equation} 

The proof is complete by applying Federer's co-area formula in the penultimate equality, that is 
\begin{equation*}
\begin{aligned}
&& \int \lVert  D^{1/2}\grad F(y)\lVert\mathbb{I}{(\vert F(y)-u\vert <\epsilon)} \calC_{F}(y)\dfrac{1}{(2\pi)^{K/2}}e^{{-\lVert y\lVert^2}/2}dy \nonumber \\
&=& \int \cfrac{\lVert  D^{1/2}\grad F(y)\lVert}{\|\grad F(y)\|}\mathbb{I}{(\vert F(y)-u\vert <\epsilon)} \calC_{F}(y)\dfrac{1}{(2\pi)^{K/2}}e^{{-\lVert y\lVert^2}/2}\|\grad F(y)\|dy \nonumber \\
&=& \int_{(u-\epsilon,u+\epsilon)}\int_{F^{-1}(z)} \frac{\|D^{1/2}\grad F(y)\|}{\|F(y)\|}\calC_{F}(y)\dfrac{1}{(2\pi)^{K/2}}e^{{-\lVert y\lVert^2}/2}d\calH_{K-1}(y)dz.
\end{aligned}
\end{equation*}
\end{proof}

This final theorem helps conclude that the coefficients in the Gaussian volume expansion of an ellipsoidal tube, introduced in \ref{GP} do match with the EC densities appearing in the expansion of the expected Euler characteristic. Thus, we complete the details of the proof of \ref{KFF-gen}.

\section{An application} \label{application}
\indent Let $\left\{ T(x),\text{ }x\in S^{2}\right\} $
denote a Gaussian, zero-mean isotropic spherical random field then it is well-known from 
(cf.\cite{marpecbook}) that the following representation holds in the
mean square sense
\begin{equation*}
y(x)=\sum_{\mathbb{\ell }m}a_{\mathbb{\ell }m}\zeta_{\mathbb{\ell }m}(x) =\sum_{%
\mathbb{\ell }} y_{\mathbb{\ell }}(x)\text{ , }\,\,\,\,\,\,\,\,\,\,y_{\mathbb{\ell }}(x)=\sum_{m=-%
\mathbb{\ell }}^{\mathbb{\ell }}a_{\mathbb{\ell }m} \zeta_{\mathbb{\ell }m}(x)%
\text{ .}
\end{equation*}%
Here $\left\{ \zeta_{\mathbb{\ell }m}(.)\right\} $ denotes the family of
spherical harmonics, and $\left\{ a_{\mathbb{\ell }m}\right\} $ the array of
random spherical harmonic coefficients, which satisfy $$\mathbb{E}a_{\mathbb{%
\ell }m}\overline{a}_{\mathbb{\ell }^{\prime }m^{\prime }}=C_{\mathbb{\ell }%
}\delta _{\mathbb{\ell }}^{\mathbb{\ell }^{\prime }}\delta _{m}^{m^{\prime
}}.$$ Further, $\delta _{a}^{b}$ is the Kronecker delta function, and the
sequence $\left\{ C_{\mathbb{\ell }}\right\} $ represents the angular power
spectrum of the field. 

\indent The random field $y$ can be shown to be almost surely continuous if the $C_{\mathbb{\ell}}$
satisfies the assumption $\sum_{\mathbb{\ell }\geq L}(2\mathbb{\ell }+1)C_{\mathbb{\ell }%
}=O(\log ^{-2}L)$. It is worth noting here that the $y_{\mathbb{\ell}}$ also represent random 
eigenfunctions of the spherical Laplacian:%
\begin{equation*}
\Delta _{S^{2}}y_{\mathbb{\ell }}=-\mathbb{\ell }(\mathbb{\ell }+1)y_{%
\mathbb{\ell }}\text{ , }\mathbb{\ell }=1,2,...
\end{equation*}%

\indent More often, spherical eigenfunctions emerge naturally from the analysis of
the Fourier components of spherical random fields. In such
cases, several (nonlinear) functionals of $y_{\mathbb{\ell }}$ assume a great
practical importance: to mention a couple, the squared norm of $T_{\mathbb{%
\ell }}$ provides an unbiased sample estimate for the angular power spectrum 
$C_{\mathbb{\ell }},$ 
\begin{equation*}
\mathbb{E}\left\{ \int_{S^{2}}T_{\mathbb{\ell }}^{2}(x)dx\right\} =(2\mathbb{%
\ell }+1)C_{\mathbb{\ell }}\text{ ,}
\end{equation*}%
while higher-order power lead to estimates of the so-called polyspectra.

\indent In the framework of cosmological data analysis (or, CMB data analysis), 
a number of papers have searched for deviations of
geometric functionals from the expected behaviour under Gaussianity. Here,
the so-called Minkowski functionals have been widely used as tools to probe
non-Gaussianity of the field $y(x)$, see \cite{matsubara2010analytic} and the references
therein. Many other works have also focussed on local deviations from the
Gaussianity assumption, mainly exploiting the properties of integrated
higher order moments (polyspectra), see \cite{pietrobon1}, \cite{rudjord2}.

\indent The univariate Gaussian kinematic formula has already found many important applications to the analysis of cosmological data, see for instance \cite{CMB-GKF, mar-vad16} and the references therein for some discussion or \cite{planck2014planck} and \cite{ade2015planck} for applications to real data. The case of multivariate spherical fields will certainly become much more important in the years to come: to mention just a possible application, we recall that most of future CMB experiments will be focussed on so-called polarization data, which can be modeled as isotropic vector-valued Gaussian fields with three components, usually labelled T, E and B modes in the cosmological literature. B modes are reckoned to be independent from the T and E components, so they fall within the framework we developed in this paper.

\indent Consider two spherical, isotropic, Gaussian random fields $y_1$ and $y_2$ such 
that they are independent, but not identically distributed, i.e., the the two fields have
different angular power spectra $C_{1,\mathbb{\ell}}$ and $C_{2,\mathbb{\ell}}$,
and as before, let the second spectral moment of $y_1$ and $y_2$ be $\lambda_1$
and $\lambda_2$, respectively.

\indent Let us consider the nonlinear subordination given by the function $F(x_1,x_2) = x_1^2 + x_2^2$. 
In the context of Gaussian tube
formula, we consider tube around 
$$F^{-1}[u,\infty) = \{(x_1,x_2)\in\real^2: x_1^2 + x_2^2 \ge u\}=\calK.$$
Since we are considering $M=S^2$, therefore, we only are interested in $\bm\calM_{l}^{\gamma_{\real^2},D}(\calK)$
for $l=0$ and $l=2$. The case $l=0$ is simple as it's just the Gaussian volume of $\calK$, implying the
only nontrivial generalised GMF we are interested in is $\bm\calM_{2}^{\gamma_{\real^2},D}(\calK)$

\noindent Using the notation of previous sections,
\begin{eqnarray} \label{eqn:M2}
&&\bm\calM_{2}^{\gamma_{\real^K},D}(\calK)=\sum_{m=0}^{1}\int_{\partial\calK} (-1)^{1-m}\times \lVert \param(x)\lVert^{1-m} \\ \nonumber
&& \;\;\;\;\;\;\;\;\;\;\;\;\;\;\;\;\;\;\;\; \times H_{1-m}\left(\left\langle x,\dfrac{D\eta_{x}}{\|D\eta_x\|}\right\rangle\right)\vp(x)\calM_{m+1}^{*}(\calK,dx)\\ \nonumber
&& \;\;\;\;\;\;\;\;\;\;\;\;\;\;\;\;\;\;\;\; = (-1) \int_{\partial\calK}  \lVert \param(x)\lVert 
\times H_{1}\left(\left\langle x,\dfrac{D\eta_{x}}{\|D\eta_x\|}\right\rangle\right)\vp(x)\calM_{1}^{*}(\calK,dx) \\ \nonumber
&& \;\;\;\;\;\;\;\;\;\;\;\;\;\;\;\;\;\;\;\; + \int_{\partial\calK} \vp(x)\calM_{2}^{*}(\calK,dx)
\end{eqnarray}
where for $\calK$ defined above, $$\calM_{m+1}^{*}(\calK,dx) = m! \,\text{detr}_m(A(x))\,\calH(dx),$$
and $\param(x) = \cfrac{D\eta_x}{\|D^{1/2}\eta_x\|}$. 
Setting $(\tilde{E}_1, \tilde{E}_{2})$ as the orthonormal basis of $\real^2$ with respect to the weighted inner product,
with $\tilde{E}_2=\param(x)$, and $\tilde{E}_1 = \frac{E_1}{\|D^{-1/2}E_1\|}$, where $E_1$ is the vector orthogonal to 
$\eta = \frac{\grad F}{\|\grad F\|}$ in the usual metric, we observe that $(\tilde{E}_1,\tilde{E}_2)$ satisfy \eqref{eqn:ONB}.
With this notation
$$
\begin{aligned}
A(z)_{11} &= \frac{1}{\|\eta_z\|_{\Qinv}} 
\langle \nabla_{\tilde{E}_{\theta}}\eta, \tilde{E}_{\theta} \rangle_{I} \biggl|_z  \\
 &= \frac{1}{\|\grad F(z)\|_{\Qinv}} 
\langle \nabla_{\tilde{E}_{1}}\grad F(z), \tilde{E}_{1} \rangle_{I} \biggl|_z
\end{aligned}
$$
Using the standard calculus, and polar coordinates to parametric the set $K$, we obtain
$$A(u,\theta)_{11} 
= \frac1{u\,\left( \lambda^{-1}_1\cos^2\theta + \lambda^{-1}_2\sin^2\theta\right)\sqrt{\lambda_1\cos^2\theta + \lambda_2\sin^2\theta}}$$

Therefore, and the above integral in equation \eqref{eqn:M2} reduces to
\begin{eqnarray*}
\bm\calM_{2}^{\gamma_{\real^K},D}(\calK) &=& (-1)  \frac{e^{-u^2/2}}{2\pi}\int_{\theta\in(0,2\pi)}  
\left( \frac{\lambda^2_1\cos^2\theta + \lambda^2_2\sin^2\theta}{\lambda_1\cos^2\theta + \lambda_2\sin^2\theta}\right)^{1/2} \\ \nonumber
&&\;\;\;\;\;\;\;\;\;\;\;\;\;\;\;\;\;\;\;\;\times (-1)\frac{\lambda_1\cos^2\theta+\lambda_2\sin^2\theta}{\sqrt{\lambda_1^2\cos^2\theta+\lambda_2^2\sin^2\theta}} \,d\theta \\ \nonumber
&& + \frac{e^{-u^2/2}}{2\pi u} \int_{\theta\in (0,2\pi)} 
\frac1{\left( \lambda^{-1}_1\cos^2\theta + \lambda^{-1}_2\sin^2\theta\right)\sqrt{\lambda_1\cos^2\theta + \lambda_2\sin^2\theta}}\,d\theta 
\end{eqnarray*}

Next, note that $\bm\calM_{0}^{\gamma_{\real^K},D}(\calK) = \gamma_{\real^K}(\calK)$, and 
for the purpose of cosmological applications the parameter space is $S^2$, hence $\bm\calL_0(S^2) = 2$, $\bm\calL_1(S^2) = 0$ and 
$\bm\calL_2(S^2) = 4\pi$.

With all this information and equation \eqref{eqn:main},
we can write a precise expression for mean Euler-Poincar\'e characteristic of excursion sets
of a random field defined as the sum of squares of two independent, but non-identically distributed Gaussian random fields.

\section{Discussion}
We conclude the paper with few questions that can be addressed in future. 
\begin{itemize}
\item 
The expected Euler characteristic draws motivation from being a good approximation to excursion probabilities of the form $\mathbb{P}(\sup_{t\in M} f(t)\geq u)$ for a wide class of smooth random fields. The statistical implication of approximating such a probability is realized in achieving a control of Family Wise error Rate (FWER) in multiple testing of statistical hypotheses at each $t\in M$. While it is known from \cite{taylor2005validity} that the Euler characteristic heuristic holds explicitly for the Gaussian case with exponentially decaying errors, that is 
$$\left |\mathbb{P}(\sup_{t\in M} f(t)\geq u)-E(\bm{\chi}(M\cap f^{-1}[u,\infty))\right |\leq C\exp(-\alpha u^2),$$
for an explicitly computed constant $\alpha$. It remains to see if such an approximation with explicit rates on the error exists for the above fields in consideration.
\item 
\cite{taylor2009gaussian} provide an elegant geometric proof of the GKF through a Poincar\'e limit approximation of the canonical isotropic process on a unit sphere in $\real^l$. More precisely, a GKF is derived by passing on limits via a Poincar\'e limit theorem from a KFF for an approximating sequence of smooth $\real^K$ valued processes on a unit sphere in $\real^l$. The approximating processes are given by
$$y^{(n)}(t,g_n)\text{ identical in distribution to } \pi_{\sqrt{n},n,k}(\sqrt{n}g_n t),$$
where $g_n\in O(n)$- the orthogonal group in dimension $n$ and $\pi_{\sqrt{n},n,k}$ denotes the projection from sphere $S_{\sqrt{n}}(\real^n)$ to $\real^k$ for $t\in S(\real^l)$. A question in the case of heterogeneous Gaussian related random fields considered in this work would be if the GKF can be re-derived through a a sequence of limiting processes, thereby enabling a classical (albeit asymptotic) geometric interpretation as in \cite{taylor2009gaussian}.
\end{itemize}

\vbox{}

\noindent{\bf Acknowledgment.} Jonathan Taylor acknowledges the support of Air Force Office of Sponsored Research grant 113039. Sreekar Vadlamani acknowledges generous support of the AIRBUS Group Corporate Foundation Chair in Mathematics of Complex Systems. SV is also thankful to Domenico Marinucci for many fruitful discussions in the lead up to this work. 

\bibliographystyle{plainnat}
\bibliography{References} 
\Addresses
\end{document}